\documentclass[a4paper,11pt,twoside,reqno]{amsart}

\RequirePackage[colorlinks,citecolor=blue,urlcolor=blue]{hyperref}
\usepackage[utf8]{inputenc}
\usepackage{tcolorbox}
\usepackage[title]{appendix}
\usepackage{empheq}
\usepackage[T1]{fontenc}
\usepackage[english]{babel}
\usepackage{cases}
\usepackage{multirow}
\usepackage{bm}
\usepackage{enumerate}
\usepackage{indentfirst}
\usepackage{xcolor}
\usepackage[top=5.cm, bottom=5.0cm, left=2cm, right=2cm]{geometry}
\usepackage{authblk}
\usepackage{comment}
\usepackage{mathtools}
\allowdisplaybreaks
\usepackage{framed}
\usepackage{amsmath}
\usepackage{amsthm}	
\usepackage{amsfonts}
\usepackage{mathrsfs}	
\usepackage{amssymb}
\usepackage{dsfont}
\usepackage{bbm}
\usepackage{tcolorbox}
\usepackage{xcolor}
\usepackage[shortlabels]{enumitem}

\theoremstyle{plain}
\newtheorem{theorem}{Theorem}[section]
\newtheorem{lemma}[theorem]{Lemma}
\newtheorem{proposition}[theorem]{Proposition}
\newtheorem{corollary}[theorem]{Corollary}
\theoremstyle{definition}
\newtheorem{definition}[theorem]{Definition}

\newtheorem{example}[theorem]{Example}

\theoremstyle{definition}
\newtheorem{remark}[theorem]{Remark}

\specialcomment{com}{}{}
\excludecomment{com} 

\numberwithin{equation}{section}
\numberwithin{figure}{section}
 \usepackage[nodayofweek]{datetime}

\usepackage[nodayofweek]{datetime}

\newcommand{\comm}[1]{}

\DeclareMathOperator*{\sign}{\text{sign}}
\newcommand{\bE}{\mathbb{E}}

\newcommand{\bN}{\mathbb{N}}
\newcommand{\bP}{\mathbb{P}}

\newcommand{\bR}{\mathbb{R}}

\def \E {\mathbb{E}}

\newcommand{\cB}{\mathcal{B}}

\newcommand{\cF}{\mathcal{F}}

\newcommand{\cI}{\mathcal{I}}

\newcommand{\cL}{\mathcal{L}}

\newcommand{\cM}{\mathcal{M}}
\newcommand{\cN}{\mathcal{N}}

\newcommand{\cP}{\mathcal{P}}

\newcommand{\cS}{\mathcal{S}}

\newcommand{\cV}{\mathcal{V}}

\def \eps {\epsilon}

\newcommand{\lev}{\left\langle}
\newcommand{\rev}{\right\rangle}

\newcommand{\N}{\mathbb{N}}%naturalnumbers
\newcommand{\R}{\mathbb{R}}	%real numbers
\newcommand{\law}[1][]{  {\mu}^{\infty}_{{#1}}  }
\newcommand{\rvlaw}[1][]{  \cL{({#1})}  }
\newcommand{\nlaw}[1][]{  {\mu}^{N}_{#1}  }

\newcommand{\indexi}[3]{  {#1}^{(#2)}_{#3}  }
\newcommand{\intrd}{\int_{\bR^d}}
\newcommand{\ld}[1][]{\frac{\delta {#1}}{\delta m}}
\newcommand{\sld}[1][]{\frac{\delta^2 {#1}}{\delta m^2}}

\newcommand{\pmu}{\partial_{\mu}}
\newcommand{\ptwomu}{\partial^2_{\mu}}
\newcommand{\pnmu}{\partial^n_{\mu}}

\title[Central limit theorem over non-linear functionals of empirical measures]{Central limit theorem over non-linear functionals of empirical measures with applications to the mean-field fluctuation of interacting diffusions}

\author{Benjamin Jourdain and Alvin Tse}
\thanks{This research benefited from the support of the ``Chaire Risques Financiers'', Fondation du Risque.}
\address{CERMICS, Ecole des Ponts, INRIA, Marne-la-Vall\'{e}e, France.}
\email{benjamin.jourdain@enpc.fr, alvin.tse@enpc.fr}

\date{}
\begin{document}	%BEGINNING
\selectlanguage{english}
\maketitle
\begin{center}
   Benjamin Jourdain and Alvin Tse \par \bigskip
\end{center}

\begin{abstract}
    In this work, a generalised version of the central limit theorem is proposed for nonlinear functionals of the empirical measure of i.i.d. random variables, provided that the functional satisfies some regularity assumptions for the associated linear functional derivative. This generalisation can be applied to Monte-Carlo methods, even when there is a nonlinear dependence on the measure component. We use this result to deal with the contribution of the initialisation in the convergence of the fluctuations between the empirical measure of interacting diffusion and their mean-field limiting measure (as the number of particles goes to infinity), when the dependence on measure is nonlinear. A complementary contribution related to the time evolution is treated using the {\it master equation}, a parabolic PDE involving $L$-derivatives with respect to the measure component, which is a stronger notion of derivative that is nonetheless related to the linear functional derivative.
\end{abstract}

{
\hypersetup{linkcolor=black}
%\tableofcontents
}

%%%%%%%%%%%%%%%%%%%%%%%%%%%%%%%%%%%%%%%%%%%%%%%%%%%%%%%%%%%%%%%%%%%%%%%%%%%%%%%%%%%%%%%%%%%%%%%%%%%%

%%%%%%%%%%%%%%%%%%%%%%%%%%%%%%%%%
%{\bf 2010 AMS subject classifications:} 
%Primary: 
%65C30% Stochastic differential and integral equations
%% 60H35%Computational methods for stochastic equations
%; secondary: 
%60H30%Applications of stochastic analysis (to PDE, etc.)
%.\\
%%%%%%%%%%%%%%%%%%%%%%%%%%%%%%%%%%%%%%%%%%%%%%%%%%%%%%%%%%
%%%%%%%%%
\section{Introduction and notations}Central limit theorems (CLTs) and their generalisations have long been studied in the last century. The first notable generalisation of the CLTs was proposed by Lyapunov in 1901, which only requires the random variables to be independent, but not necessarily identically distributed, under certain growth conditions of moments of some order $2+ \delta$. The moment condition can be further weakened in the  Lindeberg  condition  (proposed in 1922) and is  used  in  most  cases  where  weak convergence to a normal distribution is considered with non-identically distributed variables. See \cite{mether2003history} for more details regarding the history of different versions of CLTs. Since then, the literature on different types of CLTs is enormous and there are corresponding versions for dependent processes, martingales and time series. In the mathematical statistics literature, particular attention has been paid to CLTs that are uniform over a class of test functions (see for instance Sections 2.5 and 2.8 in \cite{wellner2009aad}), in order to extend the one-dimensional case of the indicator functions of the intervals $((-\infty,x])_{x\in\R}$ which is covered by the Kolmogorov-Smirnov theorem. Von Mises \cite{vonmisesaihp,vonmises} was the first to address the case of nonlinear functionals of the empirical measure $\frac{1}{N}\sum_{i=1}^N\delta_{\zeta_i}$ of independent and identically distributed $\R^d$-valued random vectors $(\zeta_i)_{i\ge 1}$ through the use of Taylor expansions and we refer to Chapter 6 in \cite{serfling} for a book presentation of the theory that he initiated. He explored the possibility that the first order term in the expansion provides a vanishing limit and then the lowest order term with nonzero limit converges to some non Gaussian distribution. While the limiting behaviour of the various terms in the expansion with derivatives computed at the common distribution $m_0$ of the random vectors may follow from standard limit theorems from probability theory (in particular, the usual central limit theorem applies to the first order contribution), the main challenge is to prove that the remainder which mixes the empirical measure with $m_0$ in the derivatives goes to zero. Let us now discuss this issue in the case treated in the present paper of first order expansions where the difficulty is not less.
In dimension $d=1$, Boos and Serfling  \cite{BoosSerfling} assume the existence of a Gateaux differential $\frac{d}{d \varepsilon} |_{\varepsilon = 0^{+}} U \big( m_0 + \varepsilon ( \nu - m_0) \big)=dU(m_0,\nu-m_0)$ (for $\nu$ any probability measure on the real line) linear in $\nu-m_0$  of $U$ at $m_0$ such that
$$U\left(\frac{1}{N}\sum_{i=1}^N\delta_{\zeta_i}\right)-U(m_0)-\frac{1}{N}\sum_{i=1}^NdU(m_0,\delta_{\zeta_i}-m_0)=o\left(\left\|\frac{1}{N}\sum_{i=1}^N 1_{\{\zeta_i\le \cdot\}}-m_0((-\infty,\cdot])\right\|_\infty\right).$$
From the boundedness in probability of $\left(\sqrt{N}\left\|\frac{1}{N}\sum_{i=1}^N 1_{\{\zeta_i\le \cdot\}}-m_0((-\infty,\cdot])\right\|_\infty\right)_{N\ge 1}$, as a  consequence of \cite{DKW}, they deduce the weak convergence of $\sqrt{N}(U(\frac{1}{N}\sum_{i=1}^N\delta_{\zeta_i})-U(m_0))$ to a centered Gaussian random variable with asymptotic variance equal to the common variance of the independent and identically distributed random variables $dU(m_0,\delta_{\zeta_i}-m_0)$ when they are square integrable and centered. For more flexibility, they remark that the conclusion remains valid when the third term in the left-hand side is multiplied by a random variable which converges in probability to $1$ as $N\to\infty$. In addition to the limitation of their approach to dimension $d=1$, it relies on the uniformity of the approximation with respect to the Kolmogorov-Smirnov distance, which is a strong assumption almost amounting to Fréchet differentiability of $U$ at $m_0$ for the Kolmogorov norm $\|\nu-m_0\|=\left\|\nu((-\infty,\cdot])-m_0((-\infty,\cdot])\right\|_\infty$:
$$U(\nu)-U(m_0)-dU(m_0,\nu-m_0)=o\left(\|\nu-m_0\|\right),$$
with $dU(m_0,\nu-m_0)$ linear but not necessarily continuous in $\nu-m_0$ for this norm. When $m_0$ is a probability on any measurable space, Dudley (\cite{Dudley}) obtains central limit theorems for $\sqrt{N}(U(\frac{1}{N}\sum_{i=1}^N\delta_{\zeta_i})-U(m_0))$ under the same notion of Fréchet differentiability with $\|\nu-m_0\|=\sup_{f\in{\mathcal F}}\left|\int f(x)(\nu-m_0)(dx)\right|$ where the class ${\mathcal F}$ of measurable functions is such that a central limit theorem for empirical measures holds with respect to uniform convergence over ${\mathcal F}$. Clearly the requirements on ${\mathcal F}$ impose some balance: the larger the value of ${\mathcal F}$, the easier Fréchet differentiability becomes, but the stronger the uniform convergence over ${\mathcal F}$ becomes. The following is mentioned by Dudley  \cite{Dudley} in p.76: ``the Gateaux derivative has been considered too weak (see also p.110 in \cite{estyetal}, p.216 in Serfling \cite{serfling}  and p.40 in Huber \cite{huber}), unless there is some uniformity along different lines and such uniformity is all the more needed in this paper''.

The linear functional derivative of $U$ (see \cite{cardaliaguet2019master}, \cite{carmona2017probabilistic}, \cite{chassagneux2019weak} and \cite{delarue2018master}) that we recall and further investigate in the second section of the present paper and subsequently apply in the third section to study the asymptotic behaviour of $\sqrt{N}(U(\frac{1}{N}\sum_{i=1}^N\delta_{\zeta_i})-U(m_0))$ is also a Gateaux derivative, but with the additional weak requirement that $dU(m_0,\nu-m_0)=\int_{\R^d}\ld[U](m_0,y)(\nu-m_0)(dy)$, for some measurable real valued function $\R^d\ni y\mapsto \ld[U](m_0,y)$ with some polynomial growth assumption in $y$. Therefore, the linearity, square integrability and centered property mentioned above (when summarizing \cite{BoosSerfling} and what we will also need) are automatically satisfied when the growth assumption is related to the index of the Wasserstein space that contains all the probability measures under consideration.
To avoid the uniformity leading to Fréchet differentiability required in the statistical literature, we suppose that the linear functional derivative exists not only at $m_0$ but on a Wasserstein ball with positive radius containing $m_0$. This is a quite mild restriction, since when a central limit theorem holds for some statistical functional, it is in general not limited to a single value of the common distribution $m_0$ of the samples. Then we linearise $\sqrt{N}(U(\frac{1}{N}\sum_{i=1}^N\delta_{\zeta_i})-U(m_0))$ into the sum of \begin{equation}
   \frac{1}{\sqrt{N}}\sum_{i=1}^N\bigg(\ld[U]\bigg(\frac{N+1-i}{N}m_0+\frac{1}{N}\sum_{j=1}^{i-1}\delta_{\zeta_j},\zeta_i\bigg)-\int_{\R^d}\ld[U]\bigg(\frac{N+1-i}{N}m_0+\frac{1}{N}\sum_{j=1}^{i-1}\delta_{\zeta_j},x\bigg)m_0(dx)\bigg)\label{mart}
\end{equation} and a remainder. This decomposition is different from the one only involving $m_0$ as the measure argument in the Gateaux derivative considered in the previously discussed literature on Von Mises differentiable statistical functions or in the recent papers \cite{delarue2018master} and \cite{szpruch2019antithetic} also using the linear functional derivative. It is aimed at enabling the analysis of the limiting behaviour of the sum by the central limit theorem for arrays of martingale increments while permitting to exploit that the very strong total variation distance between $m^{N,i}_s:=\frac{N+1-s-i}{N}m_0+\frac{1}{N}\sum_{j=1}^{i-1}\delta_{\zeta_j}+\frac{s}{N}\delta_{\zeta_i}$ and $m^{N,i}_0$ is smaller than $\frac{s}{N}$, in order to ensure that the remainder
$$ \frac{1}{\sqrt{N}} \sum_{i=1}^N \int_0^1 \intrd \left(\ld[U] (m^{N,i}_s,y)-\ld[U] (m^{N,i}_0,y)\right) \, ( \delta_{\zeta_i} - m_0)(dy) \, ds$$
vanishes in probability as $N\to\infty$, as soon as $\ld[U] (\nu,y)$ satisfies some H\"older continuity with exponent $\alpha>\frac{1}{2}$ in total variation with respect to its first variable.  In our CLT for nonlinear functionals $U$, we add some further regularity assumptions on $\ld[U]$ to check the Lindeberg condition and the convergence of the bracket of \eqref{mart}. 

The second main result of this work is a CLT on mean-field fluctuations. Large systems of interacting individuals/agents occur in many different areas of science; the individuals/agents may be people, computers, flocks of animals, or particles in moving fluid.  Mean-field theory  was developed to study particle systems by considering the asymptotic behaviour of the agents or particles, as their number goes to infinity. Instead of considering a system with a huge dimension, one can effectively approximate macroscopic and statistical features of the
system as well as the average behaviour of particles. In a probabilistic setting, the limiting behaviour can be described by a type of SDEs, called McKean-Vlasov SDEs, whose coefficients depend on the probability distribution of the process itself. We consider the fluctuation between a standard particle system $(Y^{i,N})_{1 \leq i \leq N}$ (see \eqref{eq:particlesystem} for its model) and its standard McKean-Vlasov limiting process $X$ (see \eqref{eq:MVSDE} for its equation). When the interaction only takes place in the drift coefficient and the diffusion coefficient is bounded from below (which, in particular, holds when the diffusion coefficient is constant), it is possible to express the density of the law of the interacting particle system with respect to that of independent copies of the McKean-Vlasov limiting process by Girsanov theorem. Then a central limit theorem may be derived by studying the limiting behaviour of this density using symmetric statistics and multiple Wiener integrals as in \cite{sznitman84} and \cite{shigatanaka}.

When interaction also takes place in the diffusion coefficient, this is no longer possible and the standard approach in the literature involves an approximation of the
average position of a smooth test function $\phi: \bR^d \to \bR$ of the particles by \eqref{eq:MVSDE} and its limiting fluctuation.  More precisely, denoting $\mu^N$ to be the empirical measure of all the particles and $\law[]$ to be the law of $X$, one considers the decomposition
$$ \frac{1}{N} \sum_{i=1}^N \phi(Y^{i,N}_t) = \bE \big[ \phi (X_t) \big] + \frac{1}{\sqrt{N}} \lev S^N_t, \phi \rev, $$ 
where the fluctuation measure $S^N$ is defined by
$$ S^N := \sqrt{N} \big( \nlaw[] - \law[] \big)$$
and 
$\lev m, \phi \rev := \int_{\bR^d} \phi \,d m$, for any signed measure $m$. The classical approach is to show that the sequence of random measures $(S^N)_{N \geq 1}$ converges in law as random processes taking values in some Sobolev space. This is done via a classical tightness argument, which implies the existence of a weak limit (through a subsequence) by the Prokhorov's theorem. The limit is shown to satisfy an Ornstein-Uhlenbeck process in an appropriate space. In \cite{hitsuda1986tightness}, the Sobolev space being considered is $C([0,T], \Phi'_p)$, where $\Phi'_p$ is the dual of $\Phi_p$, with $\Phi_p$ being the completion of the Schwarz space of rapidly decreasing infinitely differentiable functions under a suitable class of seminorms $\| \cdot \|_p$.  This result was generalised in \cite{meleard1996asymptotic} to the Sobolev space $C([0,T],W^{-(2+2D),D}_0)$, whereas the limiting Ornstein-Uhlenbeck process is in $C([0,T],W^{-(4+2D),D}_0)$, where $D= 1+ \left \lfloor{\frac{d}{2}}\right \rfloor $. A similar result was proven in \cite{delarue2018master} to include mean-field equations with additive common noise. We remark that, in all these approaches, by considering measures to be in the dual of a Sobolev space,   a linear dependence on the measure component is imposed implicitly. 
Unlike the approach in \cite{delarue2018master}, \cite{hitsuda1986tightness} and \cite{meleard1996asymptotic}, we analyse the fluctuation under non-linear functionals $\Phi : \cP_2(\bR^d)  \to \bR$, i.e. we consider the limiting distribution of the process
$$  F^N:= \sqrt{N}	 \big[ \Phi(\nlaw[\cdot])  - \Phi ( \law[\cdot] ) \big]$$ 
in the space $C(\bR_{+}, \bR)$,
where $\cP_2(\bR^d)$ denotes the space of probability measures with finite second moments. This gives us a limiting CLT in mean-field fluctuations in the space $C(\bR_{+}, \bR)$. 

The development of the theory in this paper relies on the calculus on the Wasserstein space. We use two notions of derivatives in measure in this paper. The first notion, the linear functional derivative, is an analogue of the variational derivative over a manifold (see \cite{cardaliaguet2019master}). Linear functional derivatives are used to prove the different versions of CLTs for i.i.d. random variables. The second notion, the L-derivative (see the notes by Cardaliaguet \cite{cardaliaguet2010notes}), was introduced by Lions in his lectures at the Coll\`ege de France  by defining a derivative in the $W_2$ space based on the `lift' to the $L^2$ space of square-integrable random variables (see \eqref{eq lift}). According to \cite{GaTu},  the L-derivative coincides with the geometric derivative introduced formerly in  \cite{ambrosio2008gradient}. L-derivatives are used to prove the CLT for mean-field fluctuations. 

The paper is organized as follows. Section 2 focuses on the notion of linear functional derivatives as well as their properties. Section 3 exhibits three versions of CLTs (with different sufficient conditions) through the properties of linear functional derivatives developed in Section 2. Finally, Section 4 develops the notion of L-derivatives followed by a version of CLT on mean-field fluctuations.

\subsection{Notations} 
$\bR_{+}$ denotes the set of non-negative real numbers. For real numbers $a$ and $b$, $a \wedge b$ and $a \vee b$ denote respectively the minimum and maximum of $a$ and $b$. For $c,d \in \bR^d$, $c \cdot d$ denotes the dot product between $c$ and $d$.  We denote the Hilbert-Schmidt norm of any matrix by $\| \cdot \|$.   For any $a,b \in \bR$ that depend on $N$, the notation $a \lesssim b$ denotes $a \leq Cb$, for some constant $C>0$ that does not depend on $N$. 

For any function $g: \bR \to \bR$, we adopt the notations $g'_{+}(s)$ or $\frac{d}{d \eps} \big|_{\eps=s^{+}} g(\eps)$ to denote the right-hand derivative of $g$ at $s \in \bR$. 
In the final section, we consider the space $C(\bR_{+}, \bR)$, which is the space of continuous functions from $\bR_{+} $ to $\bR$ equipped with the metric
$$d_{C(\bR_{+}, \bR)}(f,g):= \sum_{k=1}^{\infty} \frac{1}{2^k} \max_{1 \leq t \leq k} \Big[ | f(t) - g(t)| \wedge 1  \Big]. $$

For $\ell\ge 0$, we denote by $\cP_\ell(\bR^d)$ the set of probability measures $m$ on $\R^d$ such that $\int_{\R^d}|x|^\ell m(dx)<\infty$. For $\ell>0$, we consider the $\ell$-Wasserstein metric, defined by
\begin{eqnarray}
W_\ell ( \mu_1, \mu_2) & := & \inf \bigg\{ \bigg( \int_{\bR^d \times \bR^d} |x-y|^\ell \rho (dx , dy) \bigg)^{1/(\ell\vee 1)} \, \bigg| \,  \, \rho \in \cP_\ell ( \bR^{2d}) \, \, \text{ with } \nonumber \\
&& \rho \big( \cdot \times \bR^d) = \mu_1, \, \rho \big( \bR^d  \times \cdot) = \mu_2 \, \bigg\}, \quad \quad \mu_1, \mu_2 \in \mathcal{P}_\ell(\bR^d). \label{defwl}
\end{eqnarray}
For $\ell\ge 1$, it is well known that $W_\ell$ is a metric on $\cP_\ell(\bR^d)$ and that if $\mu\in\cP_\ell(\bR^d)$ and $(\mu_n)_{n\in\N}$ is a sequence in this space, then $\lim_{n\to\infty}W_\ell(\mu_n,\mu)=0$ iff $\mu_n$ converges weakly to $\mu$ as $n\to\infty$ and $\lim_{n\to\infty}\int_{\R^d}|x|^\ell\mu_n(dx)=\int_{\R^d}|x|^\ell\mu(dx)$ (see for instance Definition 6.4 and Theorem 6.9 in \cite{villani2008optimal}). For $\ell\in (0,1)$, the definition of $W_\ell$ is not so standard and we check in Lemma \ref{lemwl} in Appendix that these properties remain true.
We also consider the total variation metric on the set $\cP_0(\bR^d)$ of all probability measures on $\R^d$ given by
\begin{eqnarray}
W_0 ( \mu_1, \mu_2) & := & \inf \bigg\{\int_{\bR^d \times \bR^d} 1_{\{x\neq y\}} \rho (dx , dy) \, \bigg| \,  \, \rho \in \cP_0 ( \bR^{2d}) \, \, \text{ with } \nonumber \\
&& \rho \big( \cdot \times \bR^d) = \mu_1, \, \rho \big( \bR^d  \times \cdot) = \mu_2 \, \bigg\}, \quad \quad \mu_1, \mu_2 \in \mathcal{P}_0(\bR^d). \nonumber
\end{eqnarray} 
Notice that $W_0(\mu_1,\mu_2)=\sup_{A\in\cB(\R^d)}|\mu_1(A)-\mu_2(A)|=\frac{1}{2}|\mu_1-\mu_2|(\R^d)$, where $\cB(\bR^d)$ denotes the Borel $\sigma$-algebra of $\bR^d$ and $|\mu_1-\mu_2|$ the absolute value (or total variation) of the signed measure $\mu_1-\mu_2$.
We have $\inf_{\ell\ge 0}W_\ell\ge \underline W$ where $\underline W$, defined like $W_1$ but with $|x-y|\wedge 1$ replacing the integrand $|x-y|$ in \eqref{defwl}, metricizes the topology of weak convergence according to Corollary 6.13 \cite{villani2008optimal}.

For any random variable $\xi$, $\rvlaw[{\xi}]$ denotes the law of $\xi$.  Finally, $L^2(\Omega,\cF, \bP; \bR^d)$ denotes the Hilbert space of $L^2$ random variables taking values in $\bR^d$, equipped with the inner product $\lev \xi, \eta \rev = \bE[ \xi \cdot \eta]$.  
\subsection*{Acknowledgements} We thank the referee for mentionning the literature on Von Mises differentiable statistical functions to us and La\"etitia Della Maestra for numerous relevant remarks on the first version of the manuscript.
\section{Linear functional derivatives and their properties}

The notion of linear functional derivatives appears in quite a few papers in the literature. It is defined as a functional derivative in \cite{cardaliaguet2019master}, through a limit of  perturbation by linear interpolation of measures (see \eqref{eq de first order deriv}). It can also be defined via an explicit formula concerning the difference between the values of the function evaluated at two probability measures (see \eqref{difderk-1}), as more often done in the literature of mean-field games and McKean-Vlasov equations, such as \cite{carmona2017probabilistic}, \cite{chassagneux2019weak}, \cite{delarue2018master} and \cite{tse2021higher}. Corollary \ref{cordiffk-1} shows that \eqref{eq de first order deriv} implies \eqref{difderk-1} under some growth assumption. Conversely, if we assume that the linear functional derivative is continuous in the product topology of $\cP_{\ell}(\bR^d) \times \bR^d$, then one can easily check that \eqref{difderk-1} implies \eqref{eq de first order deriv}.
 
\begin{definition}
 Let $\ell\ge 0$. A function $U: \cP_\ell(\bR^d) \to \bR$ admits a linear functional derivative at $\mu\in\cP_\ell(\bR^d)$ if there exists a real valued measurable function $\R^d\ni y\mapsto \ld[U](\mu,y)$ such that $\sup_{y\in\R^d}\left|\ld[U](\mu,y)\right|/(1+|y|^\ell)<\infty$ and 
 \begin{equation}\forall \nu\in \cP_\ell(\bR^d),\; \frac{d}{d \varepsilon} \bigg|_{\varepsilon = 0^{+}} U \big( \mu + \varepsilon ( \nu - \mu) \big) = \int_{\bR^d} \frac{\delta U}{\delta m} (\mu,y) \, ( \nu - \mu)(dy). \label{eq de first order deriv} \end{equation}
 
 Inductively, for $j \geq 2$, supposing that $U$ admits a \emph{$(j-1)$-th order linear functional derivative} $(\R^d)^{j-1}\ni\mathbf{y}\mapsto\frac{\delta^{j-1} U}{\delta m^{j-1}}(m,\mathbf{y})$ at $m$ for $m$ in a $W_\ell$-neighbourhood of $\mu\in\cP_\ell(\bR^d)$, we say that $U$ admits a \emph{$j$-th order linear functional derivative} derivative at $\mu$ if for each $\mathbf{y}\in(\R^d)^{j-1}$, $m\mapsto \frac{\delta^{j-1} U}{\delta m^{j-1}}(m,\mathbf{y})$ admits a linear functional derivative at $\mu$ i.e. there exists a real valued measurable function $\R^d\ni y\mapsto \frac{\delta^{j} U}{\delta m^{j}}(\mu,\mathbf{y},y)$ such that $\sup_{y\in\R^d}\left|\frac{\delta^{j} U}{\delta m^{j}}(\mu,\mathbf{y},y)\right|/(1+|y|^\ell)<\infty$ and 
\begin{equation}\forall \nu \in \cP_2(\bR^d),\;\frac{d}{d \varepsilon} \bigg|_{\varepsilon = 0^{+}} \frac{\delta^{j-1} U}{\delta m^{j-1}}\big( \mu + \varepsilon ( \nu - \mu), \mathbf{y})  = \int_{\bR^d} \frac{\delta^{j} U}{\delta m^{j}} (\mu,\mathbf{y}, y) \, ( \nu - \mu)(dy). \label{eq de higher order deriv} \end{equation}
\end{definition} 
Notice that $W_\ell(\mu,\mu+\varepsilon(\nu-\mu))\le\varepsilon^{1\wedge 1/\ell}W_\ell(\mu,\nu)$ so that $\mu + \varepsilon ( \nu - \mu)$ belongs to the $W_\ell$-neighbourhood of $\mu$ for $\varepsilon$ small enough.
 Since $(\nu-\mu)(\R^d)=0$,  $\frac{\delta U}{\delta m}$ is defined up to an additive constant via \eqref{eq de first order deriv}. Iteratively, we normalise the higher order derivatives via the convention that 
\begin{equation}
    %\int_{\bR^d} \frac{\delta U}{\delta m}(m,y) \, m (dy) =0. \quad
   \frac{\delta^j U}{\delta m^j}(m,y_1, \ldots, y_j) = 0, \quad \text{ if } y_i=0 \, \, \text{ for some } i \in \{1, \ldots, j \}.
    \label{eq: normalisation linear functional deriatives}
\end{equation} 

% We remark the possibility of developing the theory on linear functional derivatives by replacing $\cP_2(\bR^d)$ by $\cP_1(\bR^d)$. Most of the calculations remain true for functions $U:\cP_1(\bR^d) \to \bR$. Nonetheless, the theory is presented in terms of $\cP_2(\bR^d)$ for future examples (see Example \ref{ex ii}) as well as the theory of fluctuations which uses L-derivatives.

%We shall impose the following condition on most of the results in this paper. Denoting $\frac{\delta^{0} U}{\delta m^{0}}=U$, we impose the requirement that {for every} $y_1, \ldots, y_k \in \bR^d$ {and} $m,m' \in \cP_2(\bR^d)$, 
%\begin{multline}
%\quad \quad \quad \quad \quad \quad \quad \quad  [0,1] \ni s \mapsto \frac{\delta^{k} U}{\delta m^{k}}((1-s)m +sm', y_1, \ldots , y_k) \text{ is continuous a.e.; } \\ [0,1] \ni s \mapsto \frac{\delta^{k-1} U}{\delta m^{k-1}}((1-s)m +sm', y_1, \ldots , y_{k-1}) \text{ is continuous.  }   \label{eq:k cont}   \quad \quad \quad \tag{$k$-cont} 
%\end{multline}
The following class $\cS_{j, k}(\cP_\ell(\bR^d))$ is used as  hypotheses of the central limit theorems in the subsequent section. \begin{definition}[Class $\cS_{j,k}(\cP_\ell(\bR^d))$]
For $j \in \bN$ and $k,\ell\ge 0$, the class $\cS_{j,k}(\cP_\ell(\bR^d))$ is defined by 
\begin{align*} \cS_{j,k}(\cP_\ell(\bR^d)):= \bigg\{& U: \cP_\ell(\bR^d)\to \bR \, \,  :   \text{for each } 1 \leq i \leq j , \, \, \frac{\delta^i U}{\delta m^i}\text{ exists on $\cP_\ell(\bR^d)\times(\R^d)^i$ }.\\  & \text{The map } (x_1,\ldots,x_i)\mapsto \frac{\delta^i U}{\delta m^i}(\mu,x_1, \ldots, x_i)\text{ is measurable  and} \\
&\bigg| \frac{\delta^i U}{\delta m^i}(\mu,x_1, \ldots, x_i)  \bigg| \leq C \bigg( 1+ |x_1|^{k} + \ldots + |x_i|^{k} + 1_{\{\ell>0\}}\bigg( \intrd |x|^\ell \, \mu(dx) \bigg)^{k/\ell} 
\bigg) \, ,  \\
&\text{for each } x_1, \ldots, x_i \in \bR^d \,  \text{ and } \mu \in \cP_{\ell} (\bR^d), \text{ for some } C<\infty \bigg\}.
\end{align*}
% For $\ell =0$, the class $\cS_{j,k}(\cP_0(\bR^d))$ is defined by 
% \begin{align*}  \cS_{j,k}(\cP_0(\bR^d)):= \bigg\{& U: \cP_0(\bR^d)\to \bR \, \,  :   \text{for each } 1 \leq i \leq j , \, \, \frac{ \partial^i U}{ \partial m^i} \text{ exists on $\cP_0(\bR^d)\times(\R^d)^i$ }.\\  & \text{The map } (x_1,\ldots,x_i)\mapsto \frac{ \partial^i U}{ \partial m^i}(\mu,x_1, \ldots, x_i)\text{ is measurable  and} \\
% &\bigg| \frac{ \partial^i U}{ \partial m^i}(\mu,x_1, \ldots, x_i)  \bigg| \leq C \big( 1+ 
% |x_1|^{k} + \ldots + |x_i|^{k} \big) \, ,  \\
% &\text{for each } x_1, \ldots, x_i \in \bR^d \,  \text{ and } \mu \in \cP_{0} (\bR^d), \text{ for some } C<\infty \bigg\}.
% \end{align*}
\end{definition}
The next theorem expresses a finite difference of the $(j-1)$-th order functional derivative as an integral of the $j$-th order functional derivative.
\begin{theorem} \label{lin thm} Let $\ell\ge 0$, $m,m' \in  \cP_\ell(\bR^d)$, and 
suppose that the $j$th order linear functional derivative of a function $U: \cP_\ell(\bR^d) \to \bR$ exists on the segment $(m_s:=sm'+(1-s)m)_{s\in[0,1]}$. Then for every $\mathbf{y}\in(\R^d)^{j-1}$ such that $\sup_{(s,y)\in[0,1]\times\R^d}\left|\frac{\delta^{j} U}{\delta m^{j}}(m_s,\mathbf{y},y)\right|/(1+|y|^\ell)<\infty$, the function $[0,1]\ni s\mapsto \frac{\delta^{j-1} U}{\delta m^{j-1}}(m_s,\mathbf{y})$ is Lipschitz continuous and
\begin{equation}
    \frac{\delta^{j-1} U}{\delta m^{j-1}}(m',\mathbf{y}) - \frac{\delta^{j-1} U}{\delta m^{j-1}}(m,\mathbf{y}) = \int_0^1 \int_{\bR^{d}} \frac{\delta^j U}{\delta m^j}((1-s)m +sm',\mathbf{y},y') \, (m'-m)(  dy') \, ds.\label{difderk-1}
\end{equation}
\end{theorem}
One easily deduces the following corollary.
\begin{corollary}\label{cordiffk-1}
If $U\in\cS_{j, k}(\cP_\ell(\bR^d))$ with $0\le k\le \ell$, then \eqref{difderk-1} holds for all $(m,m',\mathbf{y})\in \cP_\ell(\bR^d)\times\cP_\ell(\bR^d)\times (\R^d)^{j-1}$.
\end{corollary}
\begin{proof}[Proof of Theorem \ref{lin thm}]
For simplicity of notations, the proof is presented for $j=1$. The argument for other values of $j$ is identical. For $s\in (0,1)$ and $0<h<s\wedge (1-s)$, by the definition of linear derivatives,
\begin{align*}
    \frac{U(m_{s+h})-U(m_s)}{h}&=\frac{1}{1-s}\times\frac{U(m_s+(h/(1-s))(m'-m_s))-U(m_s)}{h/(1-s)}\\&\stackrel{h\to 0^+}\longrightarrow\frac{1}{1-s}\int_{\bR^d}\frac{\delta U}{\delta m}
(m_s,y)(1-s)(m'-m)(dy)\\
\frac{U(m_{s-h})-U(m_s)}{h}&=\frac{1}{s}\times\frac{U(m_s+(h/s)(m-m_s))-U(m_s)}{h/s}\\&\stackrel{h\to 0^+}\longrightarrow\frac{1}{s}\times\int_{\bR^d}\frac{\delta U}{\delta m}
(m_s,y)s(m-m')(dy).\end{align*}
Hence $[0,1]\ni s\mapsto U(m_s)$ is differentiable on $(0,1)$ with derivative $g(s):=\int_{\bR^d}\frac{\delta U}{\delta m}
(m_s,y)(m'-m)(dy)$, admits the right-hand derivative $g(0)$ at $0$ and the left-hand derivative $g(1)$ at $1$. This function is therefore continuous on $[0,1]$. Since $m,m' \in  \cP_\ell(\bR^d)$ and $\sup_{(s,y)\in[0,1]\times\R^d}\left|\frac{\delta U}{\delta m}(m_s,y)\right|/(1+|y|^\ell)<\infty$, the function $g$ is bounded on $[0,1]$. Therefore $[0,1]\ni s\mapsto U(m_s)$ is Lipschitz continuous. We last apply the (only) theorem in \cite{walker1977lebesgue} to deduce that
\begin{equation} \int_0^1 g(s) \, ds =   U \big( m_1\big) -  U \big( m_0\big) = U(m')- U(m). \label{fund thm right hand}  \end{equation}
\end{proof}
We now state a chain rule concerning the computation of linear functional derivatives. It is an easy consequence of the classical chain rule and the fact the normalisation convention \eqref{eq: normalisation linear functional deriatives} clearly holds.
\begin{theorem}[Chain rule]  \label{chain rule i}
Let $\ell\ge 0$, $\varphi:\cP_\ell(\bR^d) \to \bR^q$ be a function such that each of its coordinates admits a linear functional derivative at $\mu\in\cP_\ell(\bR^d)$. We denote by $\ld[\varphi](\mu,y)$ the vector in $\R^q$ with coordinates given by these linear functional derivatives. Let $F: \bR^q \to \bR$ be a function differentiable at $\varphi(\mu)$.  Then the function $U:\cP_\ell(\bR^d) \to \bR$ defined by
$U(\mu):= F(\varphi(\mu))$ admits a linear functional derivative at $\mu$ given by
$$ \frac{\delta U}{\delta m} (\mu,y) = \nabla F\big( \varphi(\mu) \big) .\ld[\varphi](\mu,y). $$ 
\end{theorem}
%\begin{proof}
%Clearly, the normalisation convention \eqref{eq: normalisation linear functional deriatives} holds. Let $\mu, \nu \in \cP_2(\bR^d)$ and let $g: [0,1] \to \bR$ be defined by
%$$ g(\eps):= \varphi( \mu+ \eps(\nu- \mu)). $$ 
%Then, by the chain rule applied to right-hand derivatives, 
%\[
 %\frac{d}{d \varepsilon} \bigg|_{\varepsilon = 0^{+}} F \big( \varphi(\mu + \varepsilon ( \nu - \mu) )\big) = (F \circ g)'_{+}(0) = F' (g(0)) g'_{+}(0) = F' (\varphi(\mu)) \intrd \ld[\varphi](\mu,y)\, (\nu- \mu)(dy).
%\]
%\end{proof}
The following example is an easy but important consequence of the chain rule and will be used in subsequent parts of the paper. 
\begin{example}[A differentiable function of a linear functional of measures]
\label{eq: simple example F G mu} 
Let $\ell\ge 0$, $G: \bR^d \to \bR$ be a measurable function such that
$$\sup_{x\in\R^d}\frac{|G(x)|}{1+|x|^\ell}<\infty$$ 
and let $F: \bR \to \bR$ be a $j$-times differentiable function. Let $U: \cP_{\ell}(\bR^d) \to \bR$ be defined by
$$ U(\mu):= F \bigg( \int_{\bR^d} G(x) \, \mu(dx)  \bigg).$$ 
Then, by Theorem \ref{chain rule i}, for $i \in \{1, \ldots, j\}$, the $i$th order linear functional derivative is given by
$$ \frac{\delta^i U}{\delta m^i} (\mu, y_1, \ldots, y_k) = F^{(i)} \bigg( \int_{\bR^d} G(x) \, \mu(dx) \bigg) \prod_{i'=1}^i (G(y_{i'})-G(0)).$$ 
Suppose that there exist constants $C>0$ and $k_i \geq 0$, $i\in \{1, \ldots, j\}$, such that
$$ |F^{(i)}(y)| \leq C(1+ |y|^{k_i}) , \quad \quad y \in \bR.$$
Then it can be checked by Young's inequality that 
$$ U \in \cS_{j, \ell\max_{1\le i\le j}\{k_i+i\}} (\cP_{\ell}(\bR^d)).$$ 
\end{example}
\begin{example}[U-statistics (see \cite{hoeffding1992class} or \cite{lee2019u}) and polynomials on the Wasserstein space] \label{exustat}
Let $k\ge 0$, $n\in{\mathbb N}$, $\varphi:(\bR^d)^n\to\R$ be measurable and such that
$$\exists C<\infty,\;\forall x_1, \ldots x_n \in \bR^d,\;\big| \varphi(x_1, \ldots, x_n) \big| \leq C(1+ |x_1|^k + \ldots + |x_n|^k).$$ 
For $\ell\ge k$, we consider the function on $\cP_{\ell} (\bR^d)$ defined by
$$ U(\mu):=  \intrd \ldots \intrd \varphi(x_1, \ldots, x_n)  \, \mu(dx_n) \ldots \, \mu(dx_1).$$
Since replacing $\varphi$ by its symmetrisation does not change the above integral, we suppose without loss of generality that $(x_1,\ldots,x_n)\mapsto \varphi(x_1,\ldots,x_n)$ is symmetric i.e. invariant by permutation of the coordinates $x_i$.
For $\mu,\nu\in \cP_\ell(\bR^d)$ and $\varepsilon\in(0,1]$, we have, denoting by $|{\mathcal N}|$ the cardinality of a subset ${\mathcal N}$ of $\{1,\ldots,n\}$,
\begin{align*}
    \frac{1}{\varepsilon}\left(U(\mu+\varepsilon(\nu-\mu))-U(\mu)\right)&=\sum_{{\mathcal N}\subset\{1,\hdots,n\}:|{\mathcal N}|\ge 1}\varepsilon^{|{\mathcal N}|-1}\int_{(\bR^d)^n}\varphi(x_1, \ldots, x_n)\bigotimes_{i\in{\mathcal N}}(\nu-\mu)(dx_i)\bigotimes_{i\in\{1,\hdots,n\}\setminus{\mathcal N}}\mu(dx_i)\\
    &\stackrel{\varepsilon\to 0^{+}}{\longrightarrow}\sum_{j=1}^n\int_{(\bR^d)^n}\varphi(x_1, \ldots, x_n)(\nu-\mu)(dx_j)\bigotimes_{i\in\{1,\hdots,n\}\setminus\{j\}}\mu(dx_i)\\
    &=\int_{\R^d}n\int_{(\R^d)^{n-1}}\varphi(y,x_2, \ldots, x_{n})\mu(dx_{n})\ldots\mu(dx_{2})(\nu-\mu)(dy),
\end{align*}
where we used the symmetry of $\varphi$ for the last equality.
Therefore $U\in\cS_{1,k}(\cP_{\ell} (\bR^d))$ with $$\ld[U](\mu,y)=n\int_{(\R^d)^{n-1}}\left(\varphi(y,x_2, \ldots, x_{n})-\varphi(0,x_2, \ldots, x_{n})\right)\mu(dx_{n})\ldots\mu(dx_{2}).$$
For $j\in\{1,\ldots,n\}$, let
$$d_j\varphi(y_1,\hdots,y_j,x_{j+1},\ldots,x_n)=\sum_{{\mathcal J}\subset\{1,\ldots,j\}}(-1)^{j-|{\mathcal J}|}\varphi(y_{\mathcal J},x_{j+1},\ldots,x_n)$$
where $y_{\mathcal J}$ denotes the vector in $(\R^d)^j$ with all coordinates with indices in ${\mathcal J}$ equal to those of $(y_1,\ldots,y_j)$ and all coordinates with indices in $\{1,\ldots,j\}\setminus{\mathcal J}$ equal to $0$. Notice that, for $i\in\{1,\hdots,j\}$, 
$$d_j\varphi(y_1,\hdots,y_j,x_{j+1},\ldots,x_n)=\sum_{{\mathcal J}\subset\{1,\ldots,j\}\setminus\{i\}}(-1)^{j-|{\mathcal J}|}\left(\varphi(y_{\mathcal J},x_{j+1},\ldots,x_n)-\varphi(y_{\mathcal J\cup\{i\}},x_{j+1},\ldots,x_n)\right)$$
and when $y_i=0$ then for each ${\mathcal J}\subset\{1,\ldots,j\}\setminus\{i\}$, $y_{\mathcal J}=y_{\mathcal J\cup\{i\}}$ so that $d_j\varphi(y_1,\hdots,y_j,x_{j+1},\ldots,x_n)=0$.
More generally, for each $j\in\N$, $U\in\cS_{j,k}(\cP_{\ell} (\bR^d))$ with
\begin{align*}
\frac{\delta^j U}{\delta m^j} (\mu,\mathbf{y},y )=\frac{n!}{(n-j)!}\int_{(\R^d)^{n-j}}d_j\varphi(\mathbf{y},y,x_{j+1}, \ldots, x_{n})\mu(dx_{n})\ldots\mu(dx_{j+1})\\
\end{align*}
when $j\le n$ and $0$ when $j>n$.

Let us suppose conversely that for some $\ell\ge 0$ and $n\ge 0$, $U\in\cS_{n+1,\ell}(\cP_{\ell} (\bR^d))$ with vanishing $\frac{\delta^{n+1} U}{\delta m^{n+1}}$. 
Then by Lemma 2.2 in \cite{chassagneux2019weak}, for $\mu,m\in\cP_{\ell} (\bR^d)$,
\begin{align*}
U(\mu) - U(m) =
\sum_{j=1}^{n}\frac1{j!} & \int_{(\bR^{d})^j} \frac{\delta^j U}{\delta m^j}(m,\mathbf{y}) \, ({\mu-m})^{\otimes j}( d \mathbf{y}) 
\\
&+ \frac1{n!}\int_0^1 (1-t)^{n} \int_{(\bR^{d})^{n+1}} \frac{\delta^{n+1} U}{\delta m^{n+1}}((1-t)m +t\mu,\mathbf{y}) \, (\mu-m)^{\otimes {(n+1)}}(d\mathbf{y}) \, dt.
\end{align*}
The assumption and the normalisation condition then give, for the choice $m=\delta_0$,
$$U(\mu)=U(\delta_0)+\sum_{j=1}^n\frac{1}{j!}\int_{(\R^d)^{j}}\frac{\delta^j U}{\delta m^j} (\delta_0,x_{1},\ldots,x_j)\mu(dx_j)\ldots \mu(dx_{1}).$$
 \begin{com}
Then we can check by backward induction that for $j\in\{1,\hdots,n\}$
\begin{align}
  \frac{\delta^j U}{\delta m^j} (\mu,y_1,\ldots,y_j)=&\frac{\delta^j U}{\delta m^j} (\delta_0,y_1,\ldots,y_j)\notag\\&+\sum_{i=j+1}^n\frac{1}{(i-j)!}\int_{(\R^d)^{i-j}}\frac{\delta^i U}{\delta m^i} (\delta_0,y_1,\ldots,y_{j},x_{j+1},\ldots,x_i)\mu(dx_i)\ldots \mu(dx_{j+1})\label{derku} 
\end{align}
and
$$U(\mu)=U(\delta_0)+\sum_{i=1}^n\frac{1}{i!}\int_{(\R^d)^{i}}\frac{\delta^i U}{\delta m^i} (\delta_0,x_{1},\ldots,x_i)\mu(dx_i)\ldots \mu(dx_{1}).$$
Indeed, by Theorem \ref{lin thm} and the normalisation condition \eqref{eq: normalisation linear functional deriatives},
\begin{align*}
    \frac{\delta^{j-1} U}{\delta m^{j-1}} (\mu,y_1,\ldots,y_{j-1})-\frac{\delta^{j-1} U}{\delta m^{j-1}} (\delta_0,y_1,\ldots,y_{j-1})
    =\int_0^1\int_{\R^d}\frac{\delta^{j} U}{\delta m^{j}} ((1-s)\delta_0+s\mu,y_1,\ldots,y_{j-1},x_j)\mu(dx_{j})ds.
\end{align*}
When $j=n+1$, the right-hand side is $0$ since $\frac{\delta^{n+1} U}{\delta m^{n+1}}$ vanishes. When \eqref{derku} holds, then, using again the normalisation condition \eqref{eq: normalisation linear functional deriatives}, we obtain that the right-hand side is equal to
\begin{align*}
    &\int_0^1\int_{\R^d}\bigg(\frac{\delta^j U}{\delta m^j} (\delta_0,y_1,\ldots,y_{j-1},x_j)\\&\phantom{\int_0^1\int_{\R^d}\bigg(}+\sum_{i=j+1}^n\frac{s^{i-j}}{(i-j)!}\int_{(\R^d)^{i-j}}\frac{\delta^i U}{\delta m^i} (\delta_0,y_1,\ldots,y_{j-1},x_{j},\ldots,x_i)\mu(dx_i)\ldots \mu(dx_{j+1})\bigg)\mu(dx_{j})ds\\&=\sum_{i=j}^n\frac{1}{(i-(j-1))!}\int_{(\R^d)^{i-(j-1)}}\frac{\delta^i U}{\delta m^i} (\delta_0,y_1,\ldots,y_{j-1},x_{j},\ldots,x_i)\mu(dx_i)\ldots \mu(dx_{j}).
\end{align*}  \makeatother
\end{com}
\end{example}

The following theorem generalizes Example \ref{exustat} by enabling a differentiable dependence of the integrand on the measure.
\begin{theorem} \label{chain rule ii}
Let $\ell\ge 0$, $\mu\in\cP_\ell(\bR^d)$ and $\varphi: (\bR^d)^n\times \cP_\ell(\bR^d) \to \bR$ be a function symmetric in its $n$ first variables such that
\begin{enumerate}[(i)]
\item for each $m\in\cP_\ell(\bR^d)$, $(\R^d)^n\ni(x_1,\ldots,x_n)\mapsto \varphi(x_1,\ldots,x_n,m)$ is measurable and integrable with respect to $m(dx_n)\ldots m(dx_1)$,
\item there exists a $W_\ell$-neighbourhood ${\mathcal N}_\mu$ of $\mu$ such that for each $(x_1,\ldots,x_n)\in(\R^d)^n$, $\cP_\ell(\bR^d)\ni m\mapsto \varphi(x_1,\ldots,x_n,m)$ admits a linear functional derivative $\ld[\varphi](x_1,\ldots,x_n,m,y)$ at $m$ for $m$ in ${\mathcal N}_\mu$ and
\begin{equation}\sup_{(m,x_1,\ldots,x_n,y)\in{\mathcal N}_\mu\times(\R^d)^{n+1}}\bigg(|\varphi(x_1,\ldots,x_n,m)|+\left|\ld[\varphi](x_1,\ldots,x_n,m,y)\right|\bigg)/(1+|x_1|^\ell+\ldots+|x_n|^\ell+|y|^\ell)<\infty.\label{quadgrowthchain2}\end{equation}
\end{enumerate}
Then the function $U:\cP_\ell(\bR^d) \to \bR$ defined by
$$ U(m):= \intrd \varphi(x_1,\ldots,x_n,m)  \, m(dx_n)\ldots m(dx_1) $$ admits a linear functional derivative at $\mu$ given by
\begin{align*}\frac{\delta U}{\delta m} (\mu,y) =& \int_{(\R^d)^{n}} \ld[\varphi](x_1,\ldots,x_n,\mu,y) +n\left(\varphi(y,x_2,\ldots,x_n,\mu) - \varphi(0,x_2,\ldots,x_n,\mu)\right)\, \mu(dx_n)\ldots\mu(dx_1) .\end{align*}
\end{theorem}
\begin{proof}
Clearly, the normalisation convention \eqref{eq: normalisation linear functional deriatives} holds. The power $\ell$ growth condition in $y$ follows from \eqref{quadgrowthchain2}.
Let $\nu\in\cP_\ell(\bR^d)$. For $\varepsilon\in(0,1]$, denoting by $|{\mathcal N}|$ the cardinality of a subset ${\mathcal N}$ of $\{1,\ldots,n\}$ as in Example \ref{exustat}, we check that the slope $\frac{1}{\varepsilon}(U(\mu+\varepsilon(\nu-\mu))-U(\mu))$ is equal to
\begin{align}&\int_{(\bR^d)^n}\frac{1}{\varepsilon}\left(\varphi(x_1, \ldots, x_n,\mu+\varepsilon(\nu-\mu))-\varphi(x_1, \ldots, x_n,\mu)\right)\mu(dx_n)\ldots\mu(dx_1)\notag\\
   &+\sum_{j=1}^n\int_{(\bR^d)^n}\varphi(x_1, \ldots, x_n,\mu+\varepsilon(\nu-\mu))(\nu-\mu)(dx_j)\bigotimes_{i\in\{1,\hdots,n\}\setminus\{j\}}\mu(dx_i)\notag\\&+
  \varepsilon\sum_{{\mathcal N}\subset\{1,\hdots,n\}:|{\mathcal N}|\ge 2}\varepsilon^{|{\mathcal N}|-2}\int_{(\bR^d)^n}\varphi(x_1, \ldots, x_n,\mu+\varepsilon(\nu-\mu))\bigotimes_{i\in{\mathcal N}}(\nu-\mu)(dx_i)\bigotimes_{i\in\{1,\hdots,n\}\setminus{\mathcal N}}\mu(dx_i).\label{decompdercomp}
\end{align}
For $\varepsilon$ small enough so that $\forall s\in[0,1]$, $\mu+s\varepsilon(\nu-\mu)\in{\mathcal N}_\mu$, by Theorem \ref{lin thm}, 
\begin{equation}
   \frac{1}{\varepsilon}\left(\varphi(x_1, \ldots, x_n,\mu+\varepsilon(\nu-\mu))-\varphi(x_1, \ldots, x_n,\mu)\right)\label{tauxaccroiss}
\end{equation} is equal to $\int_0^1\int_{\R^d}\frac{\delta \varphi}{\delta m}(x_1,\ldots,x_n,\mu+s\varepsilon(\nu-\mu),y)(\nu-\mu)(dy)ds$ and has power $\ell$ growth in $(x_1,\ldots,x_n)$ uniformly in $\varepsilon$ according to \eqref{quadgrowthchain2}. Since \eqref{tauxaccroiss} converges to $\int_{\R^d}\ld[\varphi](x_1,\ldots,x_n,\mu,y)(\nu-\mu)(dy)$ when $\varepsilon\to 0^+$, Lebesgue's theorem ensures that the first term in \eqref{decompdercomp} goes to $$\int_{(\R^d)^{n}}\int_{\R^d} \ld[\varphi](x_1,\ldots,x_n,\mu,y) (\nu-\mu)(dy)\, \mu(dx_n)\ldots\mu(dx_1).$$ By Fubini's theorem, this limit is equal to $\int_{\R^d} \int_{(\R^d)^{n}} \ld[\varphi](x_1,\ldots,x_n,\mu,y) \mu(dx_n)\ldots\mu(dx_1)(\nu-\mu)(dy)$. By Theorem \ref{lin thm}, $\varphi(x_1, \ldots, x_n,\mu+\varepsilon(\nu-\mu))$ goes to $\varphi(x_1, \ldots, x_n,\mu)$ as $\varepsilon\to 0^+$. With the growth assumption \eqref{quadgrowthchain2}, we deduce by Lebesgue's theorem that the second term in \eqref{decompdercomp} goes to
$$\sum_{j=1}^n\int_{(\bR^d)^n}\varphi(x_1, \ldots, x_n,\mu)(\nu-\mu)(dx_j)\bigotimes_{i\in\{1,\hdots,n\}\setminus\{j\}}\mu(dx_i).$$ By Fubini's theorem, symmetry of $\varphi$ in its first $n$  variables and since $(\nu-\mu)(\R^d)=0$, this limit is equal to $$\int_{\R^d}\int_{(\bR^d)^{n-1}}n\left(\varphi(y,x_2 \ldots, x_n,\mu)-\varphi(0,x_2,\ldots,x_n,\mu)\right)\mu(dx_n)\ldots\mu(dx_2)(\nu-\mu)(dy),$$
which concludes the proof.\end{proof}
The following theorem is similar to Theorem \ref{chain rule ii}, but the measure in the integral is not necessarily the same as the measure in the argument of the function $U$.
\begin{theorem}[Integral w.r.t. a different measure]\label{thmintegdifmeas}
Let $\ell\ge 0$, $\mu\in\cP_\ell(\bR^d)$, $\lambda$ be a Borel measure on $\bR^d$%  such that
% $$ \intrd C(x) \,m(dx) < +\infty,$$ 
% for some nonnegative Borel-measurable function $C: \bR^d \to \bR$. (For example, $C$ can be chosen to be 1 if $m$ is a probability measure; $e^{-|x|^2/2}$ if $m$ is the Lebesgue measure.) Let
and $\varphi: \bR^d \times \cP_\ell(\bR^d) \to \bR$ be a function such that
\begin{enumerate}[(i)]
\item for each $m\in\cP_\ell(\bR^d)$, $\R^d\ni x \mapsto \varphi(x,m)$ is Borel-measurable and integrable with respect to $\lambda$,
\item there exists a $W_\ell$-neighbourhood ${\mathcal N}_\mu$ of $\mu$ such that for each $x \in \bR^d$, $\cP_\ell(\bR^d)\ni m \mapsto \varphi(x,m)$ admits a linear functional derivative in ${\mathcal N}_\mu$ and there exists a nonnegative Borel-measurable function $C: \bR^d \to \bR$ such that
\begin{equation}\intrd C(x) \,\lambda(dx) < +\infty\mbox{ and }\sup_{(m,x,y)\in{\mathcal N}_\mu\times(\R^d)^{2}} \frac{ \big|\ld[\varphi](x,m,y) \big|}{C(x)(1+|y|^{\ell})}<\infty. \label{hypmajlam}\end{equation}
\end{enumerate}
Let $U: \cP_{\ell}(\bR^d) \to \bR$ be defined by
$$ U(m):= \intrd \varphi(x, m) \, \lambda(dx).$$ Then $U$ admits a linear functional derivative at $\mu$ given by 
$$ \ld[U](\mu,y) = \intrd \ld[\varphi](x, \mu,y) \, \lambda(dx).$$ 
\end{theorem}
\begin{proof}
We have 
$$\lim_{\eps\to 0^+}\frac{1}{\eps}\big(\varphi(x, \mu+ \eps ( \nu- \mu))- \varphi(x, \mu) \big)=\intrd \ld[\varphi](x, \mu ,y) \, (\nu- \mu)(dy).$$
Since, by Theorem \ref{lin thm}, for $\eps>0$, $$\frac{1}{\eps}\big(\varphi(x, \mu+ \eps ( \nu- \mu))-  \varphi(x, \mu) \big)=\int_0^1 \intrd \ld[\varphi](x, \mu+ s \eps(\nu-\mu),y) \, (\nu- \mu)(dy) \,ds,$$ \eqref{hypmajlam} permits to apply the dominated convergence theorem and obtain
\begin{align*}
\lim_{\eps\to 0^+} \frac{1}{\eps}& \bigg[ \intrd \varphi(x, \mu+ \eps ( \nu- \mu)) \, \lambda(dx) - \intrd \varphi(x, \mu) \, \lambda(dx) \bigg]\\&=   \intrd \intrd \ld[\varphi](x, \mu ,y) \, (\nu- \mu)(dy) \, \lambda(dx) \\
& =  \intrd \intrd \ld[\varphi](x, \mu,y) \, \lambda(dx) \, (\nu-\mu)(dy). 
\end{align*}
\end{proof}
%\subsection{Quantiles and linear functional derivatives} \label{section examples} 
Let us finally consider, in dimension $d=1$, the example of the quantile function of $m$. 
\begin{example}[quantile function]\label{lemderquant}
Let for $w\in(0,1)$ and $m\in\cP_0(\bR)$, $U(w,m):=\inf\{x\in\bR:m((-\infty,x])\ge w\}$. 
Let $v\in(0,1)$, $m_0\in\cP_0(\bR)$ be such that the restriction of $m_0$ to a neighbourhood of $U(v,m_0)$ admits a positive and continuous density $p_0$ with respect to the Lebesgue measure. Let us check that for $\nu\in\cP_0(\bR)$ such that $\nu(\{U(v,m_0)\})=0$, $$\frac{d}{d \varepsilon} \bigg|_{\varepsilon = 0^{+}} U \big(v,m_0 + \varepsilon ( \nu - m_0) \big)=-\int_{\R} \frac{1_{\{y\le U(v,m_0)\}}}{p_0(U(v,\mu))}(\nu-m_0)(dy,$$
so that, in a generalized sense related to the restriction $\nu(\{U(v,m_0)\})=0$, $\frac{\delta U}{\delta m}(v,m_0,y)=-\frac{1_{\{y\le U(v,m_0)\}}}{p_0(U(v,m_0))}$. 

Let for $\varepsilon\in[0,1]$, $m_\varepsilon:=m_0+\varepsilon(\nu-m_0)$ and $x_\varepsilon:=U(v,m_\varepsilon)$. We have 
\begin{equation}\sup_{x\in\R}|m_\varepsilon((-\infty,x))-m_0((-\infty,x))|\vee|m_\varepsilon((-\infty,x])-m_0((-\infty,x])|\le \varepsilon.\label{estunif}\end{equation}
On the neighbourhood of $x_0=U(v,m_0)$ on which $m_0$ admits a positive and continuous density, $x\mapsto m_0((-\infty,x])$ is continuously differentiable with derivative $p_0(x)$. The image of the neighbourhood by this function is a neighbourhood of $v$, on which its inverse $w\mapsto U(w,m_0)$ is also continuously differentiable with derivative $\frac{1}{p_0(U(w,m_0))}$. By \eqref{estunif} and the definition of $x_\varepsilon$, $m_0((-\infty,x_\varepsilon))\le m_\varepsilon((-\infty,x_\varepsilon))+\varepsilon\le v+\varepsilon$ and $m_0((-\infty,x_\varepsilon])\ge m_\varepsilon((-\infty,x_\varepsilon])-\varepsilon\ge v-\varepsilon$. Hence for $\varepsilon$ small enough, $m_0((-\infty,x_\varepsilon])$ and $m_0((-\infty,x_\varepsilon))$ are equal, belong to the neighbourhood of $v$ and $x_\varepsilon=U(m_0((-\infty,x_\varepsilon]),m_0)\in[U(v-\varepsilon,m_0),U(v+\varepsilon,m_0)]$ so that $\lim_{\varepsilon\to 0^+} x_\varepsilon=x_0$.

Since $m_\varepsilon((-\infty,x_\varepsilon))\le v$ and $w\mapsto U(w,m_0)$ is non-increasing, we have for $\varepsilon>0$ small enough so that $x_\varepsilon=U(m_0((-\infty,x_\varepsilon)),m_0)$
\begin{align}
    \frac{x_\varepsilon-x_0}{\varepsilon}&\le \frac{U(m_0((-\infty,x_\varepsilon)),m_0)-U(m_\varepsilon((-\infty,x_\varepsilon)), m_0)}{\varepsilon}\notag\\
    &=\frac{U(m_0((-\infty,x_\varepsilon)),m_0)-U(m_\varepsilon((-\infty,x_\varepsilon)), m_0)}{m_0((-\infty,x_\varepsilon))-m_\varepsilon((-\infty,x_\varepsilon))}(m_0-\nu)((-\infty,x_\varepsilon)),\label{majodiffin}
\end{align}
where, by convention, the first factor is equal to $\frac{\partial U}{\partial w}(v,m_0)=\frac{1}{p_0(U(v,m_0))}$ when $m_0((-\infty,x_\varepsilon))=m_\varepsilon((-\infty,x_\varepsilon))$ which is equivalent to $m_0((-\infty,x_\varepsilon))=\nu((-\infty,x_\varepsilon))$.
We have $\lim_{\varepsilon\to 0^+}m_0((-\infty,x_\varepsilon))=v$ and, by \eqref{estunif}, $\lim_{\varepsilon\to 0^+}m_\varepsilon((-\infty,x_\varepsilon))=v$. Hence, with the continuous differentiability of $w\mapsto U(w,m_0)$ in the neighbourhood of $v$, $$\lim_{\varepsilon\to 0^+}\frac{U(m_0((-\infty,x_\varepsilon)),m_0)-U(m_\varepsilon((-\infty,x_\varepsilon)), m_0)}{m_0((-\infty,x_\varepsilon))-m_\varepsilon((-\infty,x_\varepsilon))}=\frac{1}{p_0(U(v,m_0))}.$$ Since $\nu(\{x_0\})=m_0(\{x_0\})=0$, we also have $\lim_{\varepsilon\to 0^+}(m_0-\nu)((-\infty,x_\varepsilon))=(m_0-\nu)((-\infty,x_0])$ and the right-hand side of \eqref{majodiffin} converges to $\frac{(m_0-\nu)((-\infty,x_0])}{p_0(U(v,m_0))}$ as $\varepsilon\to 0^+$. 
We conclude by remarking that, since $m_\varepsilon((-\infty,x_\varepsilon])\ge v$,
$\frac{x_\varepsilon-x_0}{\varepsilon}\ge \frac{U(m_0((-\infty,x_\varepsilon]),m_0)-U(m_\varepsilon((-\infty,x_\varepsilon]), m_0)}{\varepsilon}$ where, by the same arguments, the right-hand side also converges to $\frac{(m_0-\nu)((-\infty,x_0])}{p_0(U(v,m_0))}$ as $\varepsilon\to 0^+$.

\end{example}

\section{Central limit theorem over nonlinear functionals of empirical measures} 
Let $\ell \geq 0$, $m_0\in \cP_{\ell}(\bR^d)$ and $m^N= \frac{1}{N} \sum_{i=1}^N \delta_{{\zeta}_i}$ where ${\zeta}_1, \ldots, {\zeta}_N$ are i.i.d. random variables with law $m_0$. For some nonlinear functionals $U$ on $\cP_{\ell}(\bR^d)$, we want to prove that $\sqrt{N}(U(m^N)-U(m_0))$ converges in law to some centered Gaussian random variable to generalise the result of the classical CLT which addresses linear functionals $U(\mu)= \int \varphi(x)\, \mu(dx)$ with $\varphi:\R^d\to\R$ measurable and such that $\sup_{x\in\R^d}|\varphi(x)|/(1+|x|^{\ell/2})<\infty$. Note that, by this growth assumption and Example \ref{exustat}, this linear functional belongs to $\cS_{1,  \ell/2}(\cP_{\ell}(\bR^d))$ with $\ld[U](m,x)=\varphi(x)$. For general functionals $U\in\cS_{1,  \ell/2}(\cP_{\ell}(\bR^d))$,  by the classical central limit theorem, $$\sqrt{N}  \bigg[ \intrd \ld[U](m_0,x) \, (m^N-m_0)(dx) \bigg] \stackrel{d}{\implies} \cN \bigg( 0, \text{Var}  \bigg( \frac{\delta U}{\delta m} (m_0, \zeta_1) \bigg) \bigg).$$
One could consider the same remainder \begin{equation}
   R_N:=U(m^N)-U(m_0)-\intrd \ld[U](m_0,x) \, (m^N-m_0)(dx)\label{defrn}
 \end{equation}
 as in the literature on Von Mises differentiable statistical functions  and check using a linearisation in measure by Theorem \ref{lin thm} that, under extra regularity assumptions on $U$, $\sqrt{N}R_N$ goes to $0$ in probability as $N\to\infty$. For instance,  when $U\in\cS_{4,4}(\cP_{2}(\bR^d))$ and $m_0\in\cP_{8}(\bR^d)$, Theorem 2.5 in \cite{szpruch2019antithetic} which is inspired by Lemma 5.10 in \cite{delarue2018master} ensures that $\E[R_N^2]\lesssim \frac{1}{N^2}$. In Theorem \ref{CLT measure} below, we will rather find weaker regularity assumptions under which 
 $$\sqrt{N}\bigg(\frac{1}{N}\sum_{i=1}^N\ld[U]\bigg(\frac{N+1-i}{N}m_0+\frac{1}{N}\sum_{j=1}^{i-1}\delta_{\zeta_j},\zeta_i\bigg)-\int_{\R^d}\ld[U]\bigg(\frac{N+1-i}{N}m_0+\frac{1}{N}\sum_{j=1}^{i-1}\delta_{\zeta_j},x\bigg)m_0(dx)\bigg)$$
 converges in distribution to $\cN \bigg( 0, \text{Var}  \bigg( \frac{\delta U}{\delta m} (m_0, \zeta_1) \bigg) \bigg)$ by the central limit theorem for martingale increments
 and the difference between this term and $\sqrt{N}(U(m^N)-U(m_0))$ goes to $0$ in probability as $N\to\infty$.
 
 Since the asymptotic variance is expressed in terms of $\ld[U]$, one can easily compute its value via Theorems \ref{chain rule i}, \ref{chain rule ii} and \ref{thmintegdifmeas}. % (See Section \ref{section examples} for explicit formulae.)
For functionals $U$ which do not satisfy the regularity assumptions in Theorem \ref{CLT measure}, the asymptotic variance in the central limit theorem can still be given by $\text{Var}  \left( \frac{\delta U}{\delta m} (m_0, \zeta_1) \right)$. Indeed, for the example of the quantile function in dimension $d=1$, it is shown that under the assumptions of Theorem \ref{lemderquant}, $\sqrt{N}( U(v,m^N) - U(v,m_0))$ converges in distribution to a centered Gaussian random variable with variance $\frac{v(1-v)}{p_0^2(U(v,m_0))}$. Since $1_{\{\zeta_1\le U(v,m_0)\}}$ is a Bernoulli random variable with parameter $v$ and variance $v(1-v)$, $\text{Var}  \left( \frac{\delta U}{\delta m} (v,m_0, \zeta_1) \right)=\frac{v(1-v)}{p_0^2(U(v,m_0))}$.
\begin{theorem} \label{CLT measure}
  Let $\ell \geq 0$, $m_0\in \cP_{\ell}(\bR^d)$ and $m^N= \frac{1}{N} \sum_{i=1}^N \delta_{{\zeta}_i}$, where ${\zeta}_1, \ldots, {\zeta}_N$ are i.i.d. random variables with law $m_0$.
Let $D(\mu_2,\mu_1)$ denote the metric on $\cP_{\ell}(\bR^d)$ equal to $\int_{\R^d}(1+|y|^\ell)|\mu_2-\mu_1|(dy)$ if $m_0$ is discrete and otherwise to $1_{\{\ell>0\}}W_\ell(\mu_2,\mu_1)+1_{\{\ell=0\}}\underline{W}(\mu_2,\mu_1)$.

  Suppose that there exists $r>0$ such that
\begin{itemize}
\item $U$ admits a linear functional derivative on the ball $B(m_0,r)$ centered at $m_0$ with radius $r$ for the metric $D$,
  \item \begin{equation}
   \exists C<\infty,\;\forall(\mu,x)\in B(m_0,r)\times\R^d, \quad \quad \left|\ld[U](\mu, x)
    \right|\le C\left(1+|x|^{\ell/2}\right),\label{growthball}
  \end{equation}

   \item $\exists \alpha\in(1/2,1],\;\exists C<\infty,\;\forall \mu_1,\mu_2\in B(m_0,r),\;\forall x\in\R^d,$ 
\begin{equation} \bigg| \ld[U](\mu_2, x) - \ld[U](\mu_1, x) \bigg| \leq  C\bigg((1+|x|^\ell) W_0^\alpha(\mu_2,\mu_1)+(1+|x|^{\ell(1-\alpha)})\left(\int_{\R^d}|y|^\ell|\mu_2-\mu_1|(dy)\right)^\alpha\bigg), \label{eq: Lip TV W1} \end{equation} 
    \item
    \begin{equation}
    \sup_{x\in\R^d} \frac{ \Big|\ld[U](\mu, x)-\ld[U](m_0, x) \Big|}{1+|x|^{\ell/2}} \text{ converges to } 0 \text{ when } D(\mu,m_0) \text{ goes to } 0.  \label{cont in D} 
    \end{equation}
    \end{itemize}
Then the following convergence in distribution holds :
$$ \sqrt{N} \bigg ( U(m^N) - U(m_0) \bigg) \stackrel{d}{\implies} \cN \bigg( 0, \text{Var}  \bigg( \frac{\delta U}{\delta m} (m_0, \zeta_1) \bigg) \bigg) .$$ 
% Furthermore, 
% \begin{equation} \sup_{N\in{\mathbb N}}\sqrt{N}\bE \Big| U(m^N) - U(m_0)  \Big| <\infty. \label{sec moment bound} 
% \end{equation} 
\end{theorem}

\begin{remark}Using a $W_0$-optimal coupling between $\mu_1$ and $\mu_2$, one easily checks (see for instance Theorem 6.15 \cite{villani2008optimal} when $\ell\ge 1$)  that  \begin{equation}
   W_\ell^{\ell\vee 1}(\mu_2,\mu_1)\le 2^{(\ell-1)\vee 0}W_0(\mu_1,\mu_2)\int_{\R^d}|y|^\ell|\mu_2-\mu_1|(dy)\le 2^{(\ell-1)\vee 0}\int_{\R^d}|y|^\ell|\mu_2-\mu_1|(dy).\label{majowlvarl}
 \end{equation}
Hence any ball with positive radius for the metric $1_{\{\ell>0\}}W_\ell(\mu_2,\mu_1)+1_{\{\ell=0\}}\underline{W}(\mu_2,\mu_1)$ contains a ball with positive radius for the metric $\int_{\R^d}(1+|y|^\ell)|\mu_2-\mu_1|(dy)$. Moreover \eqref{cont in D} is weaker for the latter choice of $D(\mu_1,\mu_2)$ than for the former so that the assumptions of the theorem are satisfied for the latter when they are satisfied for the former. Unfortunately, when $m_0$ is not discrete, then $\int_{\R^d}(1+|y|^\ell)|m^N-m_0|(dy)$ does not go to $0$ as $N\to\infty$. This explains why we restrict the choice $D(\mu_1,\mu_2)= \int_{\R^d}(1+|y|^\ell)|\mu_2-\mu_1|(dy)$ to the case when $m_0$ is discrete.

% On the other hand, when $\ell=0$ and
% $$\exists \alpha\in(1/2,1],\;\exists C<\infty,\;\forall \mu_1,\mu_2\in B(m_0,r),\;\forall x\in\R^d,\;\bigg| \ld[U](\mu_2, x) - \ld[U](\mu_1, x) \bigg| \leq  C\underline{W}^\alpha(\mu_2,\mu_1)$$
% then since $\underline{W}\le W_0$, both \eqref{eq: Lip TV W1} and the convergence of $\sup_{x\in\R^d}|\ld[U](\mu, x)-\ld[U](m_0, x)|/(1+|x|^{\ell/2})$ to $0$ when $\underline{W}(\mu,m_0)$ goes to $0$ hold.  In the same way, when $\ell>0$ and there exists $\alpha\in(1/2,1]$ and $C<\infty$ such that
% $$\forall \mu_1,\mu_2\in B(m_0,r),\;\forall x\in\R^d,\;\bigg| \ld[U](\mu_2, x) - \ld[U](\mu_1, x) \bigg| \leq  C(1+|x|^{\ell(1-\alpha)})W_\ell^{(\ell\vee 1)\alpha}(\mu_2,\mu_1),$$
% then, according to \eqref{majowlvarl}, both \eqref{eq: Lip TV W1} and the convergence of $\sup_{x\in\R^d}|\ld[U](\mu, x)-\ld[U](m_0, x)|/(1+|x|^{\ell/2})$ to $0$ when $W_\ell(\mu,m_0)$ goes to $0$ hold.

\end{remark}
\begin{proof}
For every $i\in\{1,\hdots,N\}$ and $s \in [0,1]$, let
\begin{equation}
   m^{N,i}_s:= \bigg( 1+ \frac{1-i-s}{N} \bigg) m_0 + \frac{1}{N} \sum_{j=1}^{i-1} \delta_{\zeta_j} + \frac{s}{N} \delta_{\zeta_i}.\label{defmnis}
 \end{equation}
 Notice that since $m_0\in\cP_{\ell}(\bR^d)$, the random measure $m^{N,i}_s$ also belongs to $\cP_{\ell}(\bR^d)$. We have $U(m^N)-U(m_0)=\sum_{i=1}^N(U(m^{N,i}_1)-U(m^{N,i}_0))$. To be able to write the difference $U(m^{N,i}_1)-U(m^{N,i}_0)$ in terms of the linear functional derivative $\ld[U]$, we are first going to check that $$\max_{1\le i\le N}\sup_{s\in[0,1]}D(m^{N,i}_s,m_0)$$ converges a.s. to $0$ as $N\to\infty$.

\noindent{\bf First step : a.s. uniform convergence of $m^{N,i}_s$ to $m_0$}\\
Since for $s\in[0,1]$, $m^{N,i}_s=sm^{N,i}_1+(1-s)m^{N,i-1}_1$ under the convention $m^{N,0}_1=m_0$, we have \begin{equation*}W_\ell^{\ell\vee 1}(m^{N,i}_s,m_0)\le sW_\ell^{\ell\vee 1}(m^{N,i}_1,m_0)+(1-s)W_\ell^{\ell\vee 1}(m^{N,i-1}_1,m_0)\le W_\ell^{\ell\vee 1}(m^{N,i}_1,m_0)\vee W_\ell^{\ell\vee 1}(m^{N,i-1}_1,m_0).\label{majinterpol}\end{equation*}
Dealing in the same way with $\underline W$, we deduce that
\begin{eqnarray}
  &&  \max_{1\le i\le N}\sup_{s\in[0,1]}\left(1_{\{\ell>0\}}W_\ell(m^{N,i}_s,m_0)+1_{\{\ell=0\}}{\underline W}(m^{N,i}_s,m_0)\right) \nonumber \\
  & = & \max_{1\le i\le N}\left(1_{\{\ell>0\}}W_\ell(m^{N,i}_1,m_0)+1_{\{\ell=0\}}{\underline W}(m^{N,i}_1,m_0)\right).\label{majoentredeux}
\end{eqnarray}
Since a.s. $m^N$ converges weakly to $m_0$ and $\int_{\R^d}|x|^\ell m^N(dx)$ goes to $\int_{\R^d}|x|^\ell m_0(dx)$ as $N\to\infty$, for $\ell>0$, the sequence $W_\ell(m^N,m_0)$ converges a.s. to $0$ as $N\to\infty$ and is therefore a.s. bounded.
Moreover,   $W_\ell(m^{N,i}_1,m_0)\le (i/N)^{1\wedge 1/\ell}W_\ell(m^i,m_0)$. For $\alpha\in(0,1)$, by considering the two cases $i/N\le \alpha$ and $i/N>\alpha$, we deduce that $$\max_{1\le i\le N}W_\ell(m^{N,i}_1,m_0)\le \alpha^{1\wedge 1/\ell}\max_{j\ge 1}W_\ell(m^j,m_0)+\max_{j\ge \lceil \alpha N\rceil}W_\ell(m^j,m_0).$$ Choosing small values of $\alpha$ followed by large values of $N$, we conclude that $\max_{1\le i\le N}W_\ell(m^{N,i}_1,m_0)$ goes to $0$  a.s.  as $N\to\infty$.

By Corollary 6.13 \cite{villani2008optimal}, $\underline W$ metricises the topology of weak convergence on ${\cP}_0(\R^d)$ and therefore $\underline W(m^N,m_0)$ converges a.s. to $0$ as $N\to\infty$. Moreover, since $|x-y|\wedge 1\le 1_{\{x\neq y\}}$, $\underline W\le W_0$ and
$\underline W(m^{N,i}_1,m_0)\le W_0(m^{N,i}_1,m_0)\le \frac{i}{N}$. % For $\alpha\in(0,1)$, by considering the two cases $i/N\le \alpha$ and $i/N>\alpha$, we deduce that $$\max_{1\le i\le N}\underline W(m^{N,i}_1,m_0)\le \alpha+\max_{j\ge \lceil \alpha N\rceil}\underline W(m^j,m_0).$$ Choosing small values of $\alpha$ followed by large values of $N$
By adapting to $\underline{W}$, as well as the above reasoning for $W_\ell$, we deduce that $\max_{1\le i\le N}\underline W(m^{N,i}_1,m_0)$ goes to $0$  a.s.  as $N\to\infty$. 

Let us now assume that $m_0$ is discrete, i.e. there is a sequence $(y_k)_{1\le k\le{\mathcal K}}$ with ${\mathcal K}\in\N^*\cup\{+\infty\}$ of distinct elements of $\R^d$ such that $\sum_{k=1}^{\mathcal K}m_0(\{y_j\})=1$. Then, by the strong law of large numbers, a.s. for each $1\le k\le {\mathcal K}$, $m^N(\{y_k\})=\frac{1}{N}\sum_{i=1}^N1_{\{\zeta_i=y_k\}}$ converges to $m_0(\{y_k\})$. When ${\mathcal K}$ is finite, we deduce that $\int_{\R^d}(1+|y|^\ell)|m^N-m_0|(dy)=\sum_{k=1}^{\mathcal K}(1+|y_k|^\ell)|m^N(\{y_k\})-m_0(\{y_k\})|$ converges to $0$ a.s. as $N\to\infty$. When ${\mathcal K}$ is infinite, we have, for each $\bar{k}\in{\mathbb N}^*$,
  \begin{align*}
   \int_{\R^d}(1+|y|^\ell)&|m^N-m_0|(dy)\le \sum_{k=1}^{\bar{k}}(1+|y_k|^\ell)|m^N(\{y_k\})-m_0(\{y_k\})|\\&+\bigg|\frac{1}{N}\sum_{i=1}^N(1+|\zeta_i|^\ell)1_{\{y_{\bar k+1},y_{\bar k+2},...\}}(\zeta_i)-\sum_{k>\bar{k}}(1+|y_k|^\ell)m_0(\{y_k\})\bigg|+2\sum_{k>\bar{k}}(1+|y_k|^\ell)m_0(\{y_k\}).
  \end{align*}
  The third term of the right-hand side is arbitrarily small for $\bar{k}$ large enough, whereas for fixed $\bar{k}$, by the strong law of large numbers, the sum of the two first terms converges a.s. to $0$ as $N\to\infty$. Hence, $\int_{\R^d}(1+|y|^\ell)|m^N-m_0|(dy)$ goes a.s. to $0$ as $N\to \infty$. Since $\int_{\R^d}(1+|y|^\ell)|m^{N,i}_1-m_0|(dy)= \frac{i}{N}\int_{\R^d}(1+|y|^\ell)|m^{i}-m_0|(dy)$, by repeating the above reasoning performed for $W_\ell$, we obtain that $\max_{1\le i\le N}\int_{\R^d}(1+|y|^\ell)|m^{N,i}_1-m_0|(dy)$ converges a.s. to $0$. Since for $s\in[0,1]$, $|m^{N,i}_s-m_0|\le s|m^{N,i}_1-m_0|+(1-s)|m^{N,i}_1-m_0|$, we conclude that $\max_{1\le i\le N}\sup_{s\in[0,1]}\int_{\R^d}(1+|y|^\ell)|m^{N,i}_s-m_0|(dy)=\max_{1\le i\le N}\int_{\R^d}(1+|y|^\ell)|m^{N,i}_1-m_0|(dy)$ converges a.s. to $0$ as $N\to\infty$.

\noindent{\bf Second step : introduction of the linear functional derivative}\\
 Under the convention $\min\emptyset:=N+1$, we deduce that for the radius $r>0$ of the ball introduced in the hypotheses of the theorem, \begin{equation}
   I_N:=\min\{1 \leq i \leq N :\exists s\in[0,1]:D(m^{N,i}_s,m_0)\ge r\}\label{defIN}
 \end{equation} is almost surely $N+1$ for each $N \geq N^{*}$, for some random variable $N^{*}$ taking integer values. For $N\ge N^*$, we have, using \eqref{growthball} and Theorem \ref{lin thm} for the second equality,
 \begin{eqnarray}
U(m^N) - U(m_0)
& = & \sum_{i=1}^N \left[ U \left( m^{N,i}_1\right) - U \left( m^{N,i}_0 \right)  \right] \nonumber \\
& = & \frac{1}{N} \sum_{i=1}^N \int_0^1 \intrd \ld[U] (m^{N,i}_s,y) \, ( \delta_{\zeta_i} - 
m_0)(dy) \, ds. \nonumber 
\end{eqnarray}
Setting 
 \begin{equation} Q_N:= \frac{1}{N} \sum_{i=1}^N \left(\intrd \ld[U](m^{N,i\wedge I_N}_0,\zeta_i)-\intrd \ld[U](m^{N,i\wedge I_N}_0,x) m_0(dx)\right), \label{defQn}
 \end{equation}
 we deduce that for $N\ge N^*$, $U(m^N)-U(m_0)-Q_N$ coincides with 
 \begin{equation} R_N:=\frac{1_{\{N\ge N^*\}}}{N} \sum_{i=1}^N \int_0^1 \intrd \left(\ld[U] (m^{N,i}_s,y)-\ld[U] (m^{N,i}_0,y)\right) \, ( \delta_{\zeta_i} - m_0)(dy) \, ds. \label{defsn}
 \end{equation}
 Therefore to check that $\sqrt{N}(U(m^N)-U(m_0)-Q_N)$  goes to $0$ in probability as $N\to\infty$, it is enough to check that so does $\sqrt{N}R_N$, which is the purpose of the third step of the proof.  By Slutsky's theorem, to complete the proof,  it is enough to check that $ \sqrt{N} Q_N  \stackrel{d}{\implies} \cN \bigg( 0, \text{Var}  \bigg( \frac{\delta U}{\delta m} (m_0, \zeta_1) \bigg) \bigg)$, which is done in the fourth step using the Central Limit Theorem for arrays of martingale increments.
 
 \noindent{\bf Third step : convergence in probability of $\sqrt{N}R_N$ to $0$}\\
 Since $|m^{N,i}_s-m^{N,i}_0|(dy)\le \frac{s}{N}(\delta_{\zeta_i}+m_0)(dy)$, using \eqref{eq: Lip TV W1} and Young's inequality, we obtain that for $N\ge N^*$,
 \begin{align*}
   \left|\ld[U](m^{N,i}_{s} ,x)-\ld[U](m^{N,i}_{0} ,x)\right|&\le C\bigg((1+|x|^\ell)\left(\frac{2s}{N}\right)^\alpha+(1+|x|^{\ell(1-\alpha)})\left(\frac{s}{N}|\zeta_i|^\ell+\frac{s}{N}\int_{\R^d}|y|^\ell m_0(dy)\right)^\alpha\bigg)% C\bigg(&(1+|x|^\ell)\left(\frac{2s}{N}\right)^\alpha+\frac{2s}{N}+\frac{s}{N}|\zeta_i|^\ell+\frac{s}{N}\int_{\R^d}|y|^\ell m_0(dy)\\&+(1+|x|^{\ell(1-\alpha)})\left(\frac{2s}{N}+\frac{s}{N}|\zeta_i|^\ell+\frac{s}{N}\int_{\R^d}|y|^\ell m_0(dy)\right)^\alpha\bigg)\
   \\
   &\le \frac{C}{N^\alpha}\left((2^\alpha+1)(1+|x|^\ell)+2|\zeta_i|^\ell+2\int_{\R^d}|y|^\ell m_0(dy)\right).
 \end{align*}
Using that $m_0\in \cP_{\ell}(\bR^d)$, we easily deduce that $\bE|R_N|\lesssim N^{-\alpha}$ and since $\alpha>1/2$, $\lim_{N\to\infty}\bE\sqrt{N}|R_N|=0$.

\noindent{\bf Fourth step : application of the Central Limit Theorem for martingales}\\
Let us introduce the filtration $({\mathcal F}_{i}:=\sigma(\zeta_1,\hdots,\zeta_{i}))_{i\ge 1}$ for which $I_N$ defined in \eqref{defIN} is a stopping time.
By \eqref{growthball}, for $1\le i\le N$, the random variable $$X_{N,i}:=\frac{1}{\sqrt{N}}\left(\intrd \ld[U](m^{N,i\wedge I_N}_0,\zeta_i)-\intrd \ld[U](m^{N,i\wedge I_N}_0,x) m_0(dx)\right)$$ is square integrable. Since $m^{N,i\wedge I_N}_0=\sum_{j=1}^{i-1}1_{\{I_n=j\}}m^{N,j}_0+1_{\{I_N>i-1\}}m^{N,i}_0$ is ${\mathcal F}_{i-1}:=\sigma(\zeta_1,\hdots,\zeta_{i-1})$-measurable and $\zeta_i$ is independent of this sigma-field, we have $\bE\left[X_{N,i}|{\mathcal F}_i\right]=0$.
Moreover,
\begin{align*}
  % \sum_{i=1}^NX^2_{N,i}&=\frac{1}{N}\sum_{i=1}^N\left(\ld[U](m^{N,i}_0,\zeta_i)-\intrd \ld[U](m^{N,i}_0,x) m_0(dx)\right)^2.\\
  \sum_{i=1}^N\bE\left[X^2_{N,i}|{\mathcal F}_{i-1}\right]&=\frac{1}{N}\sum_{i=1}^N\left(\intrd \left(\ld[U](m^{N,i\wedge I_N}_0,x)\right)^2m_0(dx)-\left(\intrd \ld[U](m^{N,i\wedge I_N}_0,x) m_0(dx)\right)^2\right).
\end{align*}
The convergence of $\sup_{x\in\R^d}\left|\ld[U](\mu, x)-\ld[U](m_0, x)\right|/(1+|x|^{\ell/2})$ to 0 when $D(\mu,m_0)$ goes to $0$ together with the a.s. convergence of $\max_{1\le i\le N}D(m^{N,i}_0,m_0)$ to $0$ imply the existence of a sequence of random variables $(\varepsilon_N)_{N\ge 0}$ converging a.s. to $0$ as $N\to\infty$ such that
\begin{equation}
   \forall 1\le i\le N,\;\left|\ld[U](m^{N,i\wedge I_N}_0,x)-\ld[U](m_0,x)\right|\le (1+|x|^{\ell/2})\varepsilon_N.\label{cvuldu}
\end{equation}
Since $m_0\in{\mathcal P}_{\ell}(\R^d)\subset{\mathcal P}_{\ell/2}(\R^d)$, we deduce that $$\max_{1\le i\le N}\left|\intrd \ld[U](m^{N,i\wedge I_N}_0,x) m_0(dx)-\intrd \ld[U](m_0,x) m_0(dx)\right|$$ converges a.s. to $0$. By continuity of the square function, so do $$\max_{1\le i\le N}\left|\left(\intrd \ld[U](m^{N,i\wedge I_N}_0,x) m_0(dx)\right)^2-\left(\intrd \ld[U](m_0,x) m_0(dx)\right)^2\right|$$ and the smaller difference of mean values $\left|\frac{1}{N}\sum_{i=1}^N\left(\intrd \ld[U](m^{N,i\wedge I_N}_0,x) m_0(dx)\right)^2-\left(\intrd \ld[U](m_0,x) m_0(dx)\right)^2\right|$. On the other hand, using \eqref{cvuldu} and \eqref{growthball}, we obtain that
\begin{align*}
   &\left|\left(\ld[U](m^{N,i\wedge I_N}_0,x)\right)^2-\left(\ld[U](m_0,x)\right)^2\right|\\&\le\left(\ld[U](m^{N,i\wedge I_N}_0,x)-\ld[U](m_0,x)\right)^2+2\left|\ld[U](m_0,x)\right|\left|\ld[U](m^{N,i\wedge I_N}_0,x)-\ld[U](m_0,x)\right|\\&\le \left((1+|x|^{\ell/2})\varepsilon_N+2C\left(1+|x|^{\ell/2}+1_{\{\ell>0\}}\left(\int_{\R^d}|y|^\ell m_0(dy)\right)^{1/2}\right)\right)(1+|x|^{\ell/2})\varepsilon_N.
\end{align*}
Since $m_0\in{\mathcal P}_{\ell}(\R^d)$, we deduce that $\max_{1\le i\le N}\left|\intrd \left(\ld[U](m^{N,i\wedge I_N}_0,x)\right)^2 m_0(dx)-\intrd \left(\ld[U](m_0,x)\right)^2 m_0(dx)\right|$ converges a.s. to $0$ and $\frac{1}{N}\sum_{i=1}^N\intrd \left(\ld[U](m^{N,i\wedge I_N}_0,x)\right)^2 m_0(dx)$ to $\intrd \left(\ld[U](m_0,x)\right)^2 m_0(dx)$. Therefore $\sum_{i=1}^N\bE\left[X^2_{N,i}|{\mathcal F}_{i-1}\right]$ converges a.s. to $\intrd \left(\ld[U](m_0,x)\right)^2 m_0(dx)-\left(\intrd \ld[U](m_0,x) m_0(dx)\right)^2=\text{Var}  \bigg( \frac{\delta U}{\delta m} (m_0, \zeta_1) \bigg) $ as $N\to\infty$. By Corollary 3.1 p58 \cite{hall2014martingale}, to conclude that $$\sqrt{N}Q_N=\sum_{i=1}^N X_{N,i}\stackrel{d}{\implies} \cN \bigg( 0, \text{Var}  \bigg( \frac{\delta U}{\delta m} (m_0, \zeta_1) \bigg) \bigg),$$ it is enough to check the Lindeberg condition :  for each $\varepsilon>0$, $\sum_{i=1}^N\E\left[X_{N,i}^21_{\{X_{N,i}^2>\varepsilon\}}|{\mathcal F}_{i-1}\right]$ goes to $0$ in probability as $N\to\infty$.
When $\ell=0$, $\ld[U]$ is bounded on $B(m_0,r)\times\R^d$ and this condition is clearly satisfied. Let us suppose that $\ell>0$ and check that it is also satisfied.

By \eqref{growthball}, 
\begin{align*}
   \exists C<\infty,\;\forall (\mu,x)\in B(m_0,r)\times\R^d,\;\left(\ld[U](\mu,x)\right)^2\le C\left(1+|x|^\ell+\int_{\R^d}|y|^\ell \mu(dy)\right).
\end{align*}
Therefore
\begin{align}
   NX_{N,i}^2&\le 2\left(\ld[U](m^{N,i\wedge I_N}_0,\zeta_i)\right)^2+2\intrd \left(\ld[U](m^{N,i\wedge I_N}_0,x)\right)^2 m_0(dx)\notag\\&\le 2C\left(2+|\zeta_i|^\ell+\frac{2}{N}\sum_{j=1}^{i-1}|\zeta_j|^\ell+3\int_{\R^d}|y|^\ell m_0 (dy)\right).\label{xni2}
\end{align}
As, for $a,b,c,d\in\R_+$, \begin{align*}
                            (a+b+c)1_{\{a+b+c\ge d\}}&= (a+b+c)\left(1_{\{a>b,a>c,a+b+c\ge d\}}+1_{\{b\ge a,b>c,a+b+c\ge d\}}+1_{\{c\ge a,c\ge b,a+b+c\ge d\}}\right)\\&\le 3a1_{\{a>b,a>c,a+b+c\ge d\}}+3b1_{\{b\ge a,b>c,a+b+c\ge d\}}+3c1_{\{c\ge a,c\ge b,a+b+c\ge d\}}\\&\le 3a1_{\{a\ge d/3\}}+3b1_{\{b\ge d/3\}}+3c1_{\{c\ge d/3\}},\end{align*}
                          it is enough to check that for each $\varepsilon>0$, $$\sum_{i=1}^N\E\left[\frac{|\zeta_i|^\ell}{N}1_{\{\frac{|\zeta_i|^\ell}{N}>\varepsilon\}}|{\mathcal F}_{i-1}\right]+\sum_{i=1}^N\E\left[\frac{1}{N^2}\sum_{j=1}^{i-1}|\zeta_j|^\ell1_{\{\frac{1}{N^2}\sum_{j=1}^{i-1}|\zeta_j|^\ell>\varepsilon\}}\bigg|{\mathcal F}_{i-1}\right]$$ goes to $0$ as $N\to\infty$. On the one hand, $\sum_{i=1}^N\E\left[\frac{|\zeta_i|^\ell}{N}1_{\{\frac{|\zeta_i|^\ell}{N}>\varepsilon\}}|{\mathcal F}_{i-1}\right]=\E\left[|\zeta_1|^\ell1_{\{|\zeta_1|^\ell>N\varepsilon\}}\right]$ goes to $0$ as $N\to\infty$ since $|\zeta_1|^\ell$ integrable. On the other hand,
\begin{align*}
   \sum_{i=1}^N\E\left[\frac{1}{N^2}\sum_{j=1}^{i-1}|\zeta_j|^\ell1_{\{\frac{1}{N^2}\sum_{j=1}^{i-1}|\zeta_j|^\ell>\varepsilon\}}\bigg|{\mathcal F}_{i-1}\right]&=\frac{1}{N^2}\sum_{i=1}^N1_{\{\frac{1}{N^2}\sum_{j=1}^{i-1}|\zeta_j|^\ell>\varepsilon\}}\sum_{j=1}^{i-1}|\zeta_j|^\ell\\&\le 1_{\{\frac{1}{N}\sum_{j=1}^{N}|\zeta_j|^\ell>N\varepsilon\}}\frac{1}{N}\sum_{j=1}^{N}|\zeta_j|^\ell,
\end{align*}
where the right-hand side goes a.s. to $0$ as $N\to\infty$, since by the strong law of large numbers, $\frac{1}{N}\sum_{j=1}^{N}|\zeta_j|^\ell$ converges a.s. to $ \int_{\R^d}|y|^\ell m_0 (dy)<\infty$.

\end{proof}

In the two next corollaries, we give sufficient conditions in terms of second order linear functional derivatives for the assumptions \eqref{eq: Lip TV W1} and \eqref{cont in D} to hold. The assumption \eqref{eq: Lip TV W1} is directly implied by the existence of a second order linear functional derivative with appropriate growth. In the case $m_0$ discrete treated in Corollary \ref{corm0disc} below, so does assumption \eqref{cont in D}, which is of similar nature since, then, $D(\mu_1,\mu_2)=\int_{\R^d}(1+|y|^\ell)|\mu_2-\mu_1|(dy)$. When $D(\mu_1,\mu_2)=1_{\{\ell>0\}}W_\ell(\mu_2,\mu_1)+1_{\{\ell=0\}}\underline{W}(\mu_2,\mu_1)$, we also suppose regularity of $\sld[U](\mu,x,y)$ with respect to $y$ to get \eqref{cont in D}.

\begin{corollary}\label{cortlcs2}Let $\ell \geq 0$, $m_0\in \cP_{\ell}(\bR^d)$ and $m^N= \frac{1}{N} \sum_{i=1}^N \delta_{{\zeta}_i}$, where ${\zeta}_1, \ldots, {\zeta}_N$ are i.i.d. random variables with law $m_0$. If $U\in{\mathcal S}_{1,\ell/2}(\cP_{\ell}(\bR^d))\cap{\mathcal S}_{2,\ell}(\cP_{\ell}(\bR^d))$ is such that $\R^d\ni y\mapsto \sld[U](\mu,x,y)/(1+|x|^{\ell/2})$ is globally H\"older continuous with exponent $\alpha\in (0,1_{\{\ell=0\}}+\ell\wedge 1]$ uniformly in $(\mu,x)\in \cP_{\ell}(\bR^d)\times\R^d$, then 
   $$ \sqrt{N} \bigg ( U(m^N) - U(m_0) \bigg) \stackrel{d}{\implies} \cN \bigg( 0, \text{Var}  \bigg( \frac{\delta U}{\delta m} (m_0, \zeta_1) \bigg) \bigg) .$$
 Furthermore, 
\begin{equation} \sup_{N\in{\mathbb N}}\sqrt{N}\bE \Big| U(m^N) - U(m_0)  \Big| <\infty. \label{sec moment bound} 
\end{equation} \end{corollary}
\begin{proof}
  Let us check for the first statement that the hypotheses of Theorem \ref{CLT measure} are satisfied. Since $U \in{\mathcal S}_{1,\ell/2}(\cP_{\ell}(\bR^d))$, \eqref{growthball} holds.%  Since $ U \in{\mathcal S}_{2,\ell/2}(\cP_{\ell}(\bR^d))$, then by Corollary \ref{cordiffk-1}, for $\mu_1,\mu_2\in{\mathcal P}_\ell(\R^d)$ and $x\in\R^d$,
  % \begin{align*}
  %  \ld[U](\mu_2,x)-\ld[U](\mu_1,x)=\int_0^1\int_{\R^d}\sld[U](s\mu_2+(1-s)\mu_1,x,y)(\mu_2-\mu_1)(dy)ds.
  % \end{align*}
  % If $C$ denotes the uniform in $(\mu,x)$ Lipschitz constant of $y\mapsto\sld[U](\mu,x,y)$, the dual formulation of $W_1$ implies that for all $s\in[0,1]$,
  % $\left|\int_{\R^d}\sld[U](s\mu_2+(1-s)\mu_1,x,y)(\mu_2-\mu_1)(dy)\right|\le CW_1(\mu_2,\mu_1)$ so that $\left|\ld[U](\mu_2,x)-\ld[U](\mu_1,x)\right|\le CW_1(\mu_2,\mu_1)$.
  
  If $\ell=0$, $U\in\cS_{2,0}(\cP_{0}(\bR^d))$ and $\sld[U]$ is bounded by some finite constant $C$ so that for all $\mu_1,\mu_2\in{\mathcal P}_0(\R^d)$ and $x\in\R^d$,
\begin{align*}
   \bigg| \ld[U](\mu_2, x) - \ld[U](\mu_1, x) \bigg| =\left|\int_0^1\int_{\R^d}\sld[U](s\mu_2+(1-s)\mu_1, x,y)(\mu_2-\mu_1)(dy)ds\right|\le CW_0(\mu_2,\mu_1)
\end{align*}
and \eqref{eq: Lip TV W1} is satisfied for $\ell=0$.\\If $\ell>0$, $U\in\cS_{2,\ell}(\cP_{\ell}(\bR^d))$ and
$$\exists C<\infty,\;\forall (\mu,x,y)\in {\mathcal P}_\ell(\R^d)\times\R^d\times\R^d,\;\left|\sld[U](\mu, x,y)\right|\le C\left(1+|x|^\ell+|y|^\ell+\int_{\R^d}|z|^\ell\mu(dz)\right).$$
For $m_0\in {\mathcal P}_\ell(\R^d)$ and $\mu\in B(m_0,r)$ the ball centered at $m_0$ with $W_\ell$ radius $r$ ,
$$\int_{\R^d}|z|^\ell\mu(dz)\le 2^{(\ell-1)\vee 0}\left(\int_{\R^d}|z|^\ell m_0(dz)+W_\ell^{\ell\vee 1}(m_0,\mu)\right)\le 2^{(\ell-1)\vee 0}\left(\int_{\R^d}|z|^\ell m_0(dz)+r^{\ell\vee 1}\right).$$
Let $\mu_1,\mu_2\in B(m_0,r)$ and $x\in\R^d$. Since for $s\in[0,1]$, $s\mu_2+(1-s)\mu_1\in B(m_0,r)$, we deduce that
\begin{align}
   \bigg| \ld[U](\mu_2, x) - \ld[U](\mu_1, x) \bigg| \le &C\left(1+2^{(\ell-1)\vee 0}\left(\int_{\R^d}|z|^\ell m_0(dz)+r^{\ell\vee 1}\right)+|x|^\ell\right)W_0(\mu_2,\mu_1)\notag\\&+C\int_{\R^d}|y|^\ell|\mu_2-\mu_1|(dy),\label{majdifldmu}
\end{align}
so that \eqref{eq: Lip TV W1} is satisfied with $\alpha =1$.
Still for $\ell>0$, introducing a $W_\ell$-optimal coupling $\pi$ between $\mu_1$ and $\mu_2$, we deduce from the H\"older continuity property that
\begin{align}
  \bigg| \ld[U](\mu_2, x) &- \ld[U](\mu_1, x) \bigg|\notag\\&=\left|\int_0^1\int_{\R^d\times\R^d}\left(\sld[U](s\mu_2+(1-s)\mu_1, x,y)-\sld[U](s\mu_2+(1-s)\mu_1, x,z)\right)\pi(dy,dz)ds\right|\notag\\
  &\le C(1+|x|^{\ell/2})\int_{\R^d\times\R^d}|y-z|^\alpha\pi(dy,dz)\le C(1+|x|^{\ell/2})W^{\alpha/(\ell\wedge 1)}_{\ell\wedge 1}(\mu_2,\mu_1)\notag\\&\le C(1+|x|^{\ell/2})W^{\alpha/(\ell\wedge 1)}_{\ell}(\mu_2,\mu_1).\label{majodifldu}
\end{align}
When $\ell=0$, introducing a $\underline{W}$-optimal coupling $\pi$ between $\mu_1$ and $\mu_2$ and  also using the boundedness of $\sld[U]$, we get
\begin{equation}
 \bigg| \ld[U](\mu_2, x) - \ld[U](\mu_1, x) \bigg|\le C\int_{\R^d\times\R^d}(|y-z|^\alpha\wedge 1) \pi(dy,dz)\le C\underline{W}^\alpha(\mu_2,\mu_1).\label{majodifldubis}  
\end{equation}
Using a $W_0$-optimal coupling between $\mu_1$ and $\mu_2$, one easily checks (see, for instance, Theorem 6.15 in \cite{villani2008optimal} when $\ell\ge 1$)  that  \begin{equation*}
   W_\ell^{\ell\vee 1}(\mu_2,\mu_1)\le 2^{(\ell-1)\vee 0}\int_{\R^d}|y|^\ell|\mu_2-\mu_1|(dy).
 \end{equation*}
 Since $\underline{W}(\mu_2,\mu_1)\le W_0(\mu_2,\mu_1)=\int_{\R^d}|\mu_2-\mu_1|(dy)$, we deduce that any ball for the metric $D$ is included in a ball for $1_{\{\ell>0\}}W_\ell+1_{\{\ell=0\}}\underline{W}$ and that the convergence to $0$ of $D(\mu,m_0)$ implies that of $1_{\{\ell>0\}}W_\ell(\mu,m_0)+1_{\{\ell=0\}}\underline{W}(\mu,m_0)$ and of $\sup_{x\in\R^d}|\ld[U](\mu, x)-\ld[U](m_0, x)|/(1+|x|^{\ell/2})$ by \eqref{majodifldu} and \eqref{majodifldubis}.

For the second statement, we may choose $r=+\infty$ in the proof of Theorem \ref{CLT measure}, so that $I_N=N+1$ in \eqref{defQn} and $N^*=0$ in \eqref{defsn} define the two terms in the decomposition $U(m^N)-U(m_0)=Q_N+R_N$. From the third step of that proof, we have $\bE|R_N|\lesssim N^{-\alpha}$ with $\alpha>1/2$,  while the martingale property and the estimation \eqref{xni2} ensure that $\sup_{N\ge 1}N\E[Q_N^2]\le\sup_{(i,N):1\le i\le N}N\E[X_{N,i}^2]<\infty$. 
\end{proof}\begin{corollary}\label{corm0disc}Let $\ell \geq 0$, $m_0\in \cP_{\ell}(\bR^d)$ be discrete and $m^N= \frac{1}{N} \sum_{i=1}^N \delta_{{\zeta}_i}$, where ${\zeta}_1, \ldots, {\zeta}_N$ are i.i.d. random variables with law $m_0$. If $U\in{\mathcal S}_{1,\ell/2}(\cP_{\ell}(\bR^d))\cap{\mathcal S}_{2,\ell}(\cP_{\ell}(\bR^d))$ is such that$$\exists C<\infty,\;\forall (\mu,x,y)\in {\mathcal P}_\ell(\R^d)\times\R^d\times\R^d,\;\left|\sld[U](\mu, x,y)\right|\le C\left(1+|x|^{\ell/2}+|y|^\ell+\int_{\R^d}|z|^\ell\mu(dz)\right),$$  then 
   $$ \sqrt{N} \bigg ( U(m^N) - U(m_0) \bigg) \stackrel{d}{\implies} \cN \bigg( 0, \text{Var}  \bigg( \frac{\delta U}{\delta m} (m_0, \zeta_1) \bigg) \bigg) .$$
 Furthermore, 
\begin{equation*} \sup_{N\in{\mathbb N}}\sqrt{N}\bE \Big| U(m^N) - U(m_0)  \Big| <\infty.
\end{equation*} \end{corollary}
Notice that the assumptions on $U$ are satisfied as soon as $U\in{\mathcal S}_{2,\ell/2}(\cP_{\ell}(\bR^d))$.
\begin{proof}
The only difference with Corollary \ref{cortlcs2} concerns the proof that $\sup_{x\in\R^d}|\ld[U](\mu, x)-\ld[U](m_0, x)|/(1+|x|^{\ell/2})$ goes to $0$ as $D(\mu,m_0)=\int_{\R^d}(1+|x|^\ell)|\mu-m_0|(dx)$ goes to $0$. This continuity property is implied by the fact that, under  the growth assumption on $\sld[U](\mu, x,y)$, \eqref{majdifldmu} holds with $|x|^{\ell}$ replaced by $|x|^{\ell/2}$ in the right-hand side.
\end{proof}
The following example illustrates the power of Corollary \ref{cortlcs2}, if a function behaves badly w.r.t. the measure component, but is very regular w.r.t. the spatial components. In this case, the conditions in Corollary \ref{cortlcs2} are easier to verify than Theorem \ref{CLT measure}.
\begin{example}[Conditions in Corollary \ref{cortlcs2} are satisfied]
Let $U: \cP_0(\bR) \to \bR$ be defined by
$$ U (\mu):= \bigg| \int_{\bR} \sin(y) \, \mu(dy) \bigg|^{5/2}  . $$ 
By Example  \ref{eq: simple example F G mu},
$$ \ld[U](\mu,x_1) = \frac{5}{2} \bigg| \int_{\bR} \sin(y) \, \mu(dy) \bigg|^{3/2} \sign \bigg( \int_{\bR} \sin(y) \, \mu(dy) \bigg) \sin (x_1), $$
where $\sign: \bR \to \bR$ is the function defined by
$$ \sign (x):= \begin{cases} 
      -1, & x < 0, \\
      0, & x=0, \\
      1, & x>0,
   \end{cases} $$ 
and
$$ \sld[U](\mu,x_1,x_2) = \frac{15}{4} \bigg| \int_{\bR} \sin(y) \, \mu(dy) \bigg|^{1/2} \sin (x_1)  \sin (x_2).$$ 
Clearly, $U \in \cS_{1,0} (\cP_0(\bR)) \cap \cS_{2,0} (\cP_0(\bR))$. Moreover, $x_2 \mapsto \sld[U](\mu,x_1,x_2)$ is Lipschitz continuous, uniformly in $\mu$ and $x_1$. Therefore, the CLT holds for $U$ by Corollary \ref{cortlcs2} applied with $\ell=0$. Of course, it can also be deduced from the classical delta method.
\end{example}

We note that since Theorem  \ref{CLT measure} is more general than Corollary \ref{cortlcs2}, there are examples where  conditions of  Theorem  \ref{CLT measure} hold but not Corollary \ref{cortlcs2}. 
\begin{example} [Conditions in Theorem \ref{CLT measure} hold but not  Corollary \ref{cortlcs2}] \label{ex ii} 
Let $U: \cP_{12}(\bR) \to \bR$ be defined by
$$ U(\mu):= \bigg( \int_{\bR} x^2 \, \mu(dx) \bigg)^3.$$ 
By Example \ref{eq: simple example F G mu}, $\ld[U]$ $\sld[U]$ and $\frac{ \delta^3 U}{\delta m^3}$ all exist and are given by
$$ \ld[U](\mu,y_1)= 3 \bigg( \int_{\bR} x^2 \, \mu(dx) \bigg)^2 y_1^2, $$ 
$$ \quad \quad \quad \sld[U](\mu,y_1, y_2) = 6 \bigg( \int_{\bR} x^2 \, \mu(dx) \bigg)  y_1^2  y_2^2 , $$ 
and
$$ \frac{ \delta^3 U}{\delta m^3}(\mu, y_1, y_2, y_3) = 6  y_1^2  y_2^2 y_3^2.$$ 
By Young's inequality, $U \in \cS_{3,6} (\cP_{12}(\bR))$. Therefore,  $U \in \cS_{1,6} (\cP_{12}(\bR)) \cap \cS_{2,12} (\cP_{12}(\bR)) $. However, the condition on H\"older continuity in Corollary \ref{cortlcs2} does not hold, since $y \mapsto y^2$ is not uniformly continuous on $\bR$, therefore it cannot be H\"older continuous.

We now show that the conditions of Theorem \ref{CLT measure} hold for $\ell=12$ by showing that \eqref{eq: Lip TV W1} and \eqref{cont in D} are satisfied. Pick $r >0$ and consider the ball $B(m_0, r)$ in the $W_{12}$ metric for $m_0\in\cP_{12}(\R)$. Since $W_2\le W_{12}$, there exists a constant $C>0$ (depending on $r$ and $m_0$) such that
$$\int_{\bR} x^2 \, \mu (dx)  \leq C, \quad \quad \forall \mu \in B(m_0,r). $$ 
This implies that, for every $\mu_1, \mu_2 \in B(m_0,r)$,
\begin{align*}
  \bigg| \ld[U](\mu_2, x) -  \ld[U](\mu_1, x)  \bigg| &\leq 6 C\int_{\bR} x^2y^2|\mu_2-\mu_1|(dy)\le 3C x^4W_0(\mu_2,\mu_1)+3C\int_{\bR} y^4|\mu_2-\mu_1|(dy)\\&\le 3C(1+x^{12})W_0(\mu_2,\mu_1)+3C\int_{\bR}(1+y^{12})|\mu_2-\mu_1|(dy)\\&\le 6C(1+x^{12})W_0(\mu_2,\mu_1)+3C\int_{\bR}y^{12}|\mu_2-\mu_1|(dy)
\end{align*}
which proves \eqref{eq: Lip TV W1} for $\alpha=1$. Finally, we recall from Theorem 5.5 of \cite{carmona2017probabilistic} that
$W_2(\mu, m_0) \to 0$ implies that
$$ \int_{\bR} x^2 \, \mu  (dx) \to \int_{\bR} x^2  \, m_0 (dx). $$ 
Consequently, as $W_{12}(\mu,m_0)$ and therefore $W_2(\mu,m_0)$ go to $0$,
\begin{align*}
\sup_{x\in\R} \frac{ \Big|\ld[U](\mu, x)-\ld[U](m_0, x) \Big|}{1+|x|^{6}}  
& =  \sup_{x\in\R} \frac{3 \Big[ \big( \int_{\bR} y^2 \, \mu(dy) \big)^2 - \big( \int_{\bR} y^2 \, m_0(dy) \big)^2 \Big] x^2  }{1+|x|^{6}} \nonumber \\
& \leq 3 \Big[ \Big( \int_{\bR} y^2 \, \mu(dy) \Big)^2 - \Big( \int_{\bR} y^2 \, m_0(dy) \Big)^2 \Big] \to 0, \nonumber 
\end{align*}
which proves \eqref{cont in D}. 
\end{example}

\section{An application in mean-field theory: fluctuations of interacting diffusion over nonlinear functionals of measures} 
 \subsection{L-derivatives}
 In this section, we introduce the notion proposed by P.-L. Lions, which was expounded in other works in the literature (e.g. \cite{buckdahn2017mean, cardaliaguet2010notes,carmona2017probabilistic,crisan2017smoothing}).  Combining this notion with the notion of linear functional derivatives, we present the analysis of fluctuations of McKean-Vlasov SDEs over nonlinear functionals of measures. 

Suppose that the probability space $(\Omega, \cF, \bP)$ is atomless (i.e. there does not exist a measurable set which has positive measure and contains no set of smaller positive measure). Then for any $\mu \in \cP_0(\bR^d)$, we can always construct an $\bR^d$-valued random variable on $\Omega$ with law $\mu$ (see page 376 from \cite{carmona2017probabilistic}).

 For any function $U: \cP_2(\bR^d) \to \bR$, we define the lift $\widetilde{U}: L^2(\Omega,\cF, \bP; \bR^d) \to \bR$ by
\begin{equation} \widetilde{U}(\theta):= U(\rvlaw[\theta]). \label{eq lift} \end{equation}
Recall that $\widetilde{U}$ is said to the Fr\'{e}chet differentiable at $\theta_0$ if there exists a linear continuous map $D\widetilde{U}(\theta_0): L^2(\Omega,\cF, \bP; \bR^d) \to \bR$ such that
$$ \widetilde{U}(\theta_0 + \eta) - \widetilde{U}(\theta_0) = D\widetilde{U}(\theta_0)(\eta) + o(\| \eta\|_{L^2}),$$
as $\|\eta\|_{L^2} \to 0$. By the Riesz representation theorem, there exists a ($\bP$-a.s.) unique random variable $L_{\theta_0} \in L^2(\Omega,\cF, \bP; \bR^d)$ such that
$$D\widetilde{U}(\theta_0)(\eta) = \bE[ L_{\theta_0} \eta], \quad \quad \forall \eta \in L^2(\Omega,\cF, \bP; \bR^d). $$ 
The following theorem follows from Theorem 6.2 and Theorem 6.5 from \cite{cardaliaguet2010notes} (or equivalently, Proposition 5.24 and Proposition 5.25 from \cite{carmona2017probabilistic}) combined with Corollary 3.22 \cite{GaTu}.
\begin{theorem} \label{thm structure of gradient L derivatives} 
Suppose that $\widetilde{U}$ is Fr\'{e}chet differentiable at $\theta_0$ and $\hat{\theta}_0$. Suppose that 
$ \rvlaw[\theta_0] = \rvlaw[\hat{\theta}_0] = \mu \in \cP_2(\bR^d)$. Then
\begin{enumerate}[(i)]
    \item The joint law $(\theta_0,L_{\theta_0})$ is equal to the joint law of $(\hat{\theta}_0, L_{\hat{\theta}_0})$.
    \item \label{borel measurable structure of gradient} There exists a Borel-measurable function $h : \bR^d \to \bR^d$ (uniquely determined $\mu$-a.e.) such that
    $ \int_{\bR^d} | h(x)|^2 \, \mu(dx) < + \infty$ and 
    $$ h(\theta_0) = L_{\theta_0}, \quad \quad h(\hat{\theta}_0)= L_{\hat{\theta}_0}, \quad \quad \text{a.s.} $$ 
\end{enumerate}
\end{theorem}
We are now in a position to define L-derivatives. The previous theorem tells us that the following definition makes sense. 
\begin{definition}
\begin{enumerate}[(i)]
\item A function $U: \cP_2(\bR^d) \to \bR$ is said to be L-differentiable at $\mu \in \cP_2(\bR^d)$ if there exists a random variable $\theta_0$ with law $\mu$ such that $\widetilde{U} $ is Fr\'{e}chet differentiable at $\theta_0$. 
\item If $U: \cP_2(\bR^d) \to \bR$ is L-differentiable at $\mu \in \cP_2(\bR^d)$, then its L-derivative \footnote{For brevity, in this work, we say \emph{the} L-derivative, rather than a $\mu$-version of L-derivative. Any property imposed on the L-derivatives in later parts means that it is applicable to at least one $\mu$-version.} $\pmu U(\mu)$ is defined to be $\pmu U(\mu):= h$, where $h: \bR^d \to \bR^d$ is the Borel-measurable function in  \eqref{borel measurable structure of gradient} of Theorem \ref{thm structure of gradient L derivatives}. Moreover, we define the joint map $\pmu U: \cP_2(\bR^d) \times \bR^d \to \bR^d$ by
$$ \pmu U(\mu,y):= [ \pmu U(\mu)](y).$$ 
\end{enumerate}
\end{definition}
We define higher order derivatives of measure functionals by iterating the definitions of L-derivatives. Following the approach adopted in the work \cite{chassagneux2019weak} and \cite{crisan2017smoothing}, for any $k \in \bN$, we formally define higher order derivatives in measures through the following iteration (provided that they actually exist): for any $k \geq 2$, $(i_1, \ldots, i_k) \in \{ 1, \ldots, d \}^k$ and $x_1, \ldots, x_k \in \bR^d$, the function $\partial^k_{\mu} f:\mathcal{P}_2(\bR^d) \times (\bR^d)^{ k} \to (\bR^d)^{\otimes k}$ is defined by
\begin{equation} \bigg( \partial^k_{\mu} f( \mu, x_1, \ldots, x_k) \bigg)_{(i_1, \ldots, i_k)} := \bigg( \pmu \bigg( \Big( \partial^{k-1}_{\mu} f( \cdot, x_{1}, \ldots, x_{k-1}) \Big)_{(i_1, \ldots, i_{k-1})} \bigg)(\mu,x_k) \bigg)_{i_k}, \label{eq:generalformulaintro} \end{equation} 
and its corresponding mixed derivatives in space $\partial^{\ell_k}_{v_k} \ldots \partial^{\ell_1}_{v_1} \partial^k_{\mu} f:\mathcal{P}_2(\bR^d) \times (\bR^d)^{ k} \to (\bR^d)^{\otimes (k+ \ell_1 +\ldots \ell_k)} $ are defined by
\begin{equation} \bigg( \partial^{\ell_k}_{v_k} \ldots \partial^{\ell_1}_{v_1} \partial^k_{\mu} f( \mu, x_1, \ldots, x_k) \bigg)_{(i_1, \ldots, i_k)} := \frac{\partial^{\ell_k}}{\partial x^{\ell_k}_k} \ldots \frac{\partial^{\ell_1}}{\partial x^{\ell_1}_1} \bigg[ \bigg( \partial^k_{\mu} f( \mu, x_1, \ldots, x_k) \bigg)_{(i_1, \ldots, i_k)} \bigg],  \label{eq:generalformulamixed} \end{equation}
for $\ell_1 \ldots \ell_k \in \bN \cup \{0 \}$. The spatial derivatives commute with the derivatives in measure as long as $j$ derivatives in the measure are kept at the right of each $\partial_{v_j}$. Since this notation for higher order derivatives in measure is quite cumbersome, we introduce the following multi-index notation for brevity. 
%We also introduce the class $\cM_k$ of $k$th order differentiable functions, which is needed to present subsequent results in this section. 
\begin{definition}[Multi-index notation]
Let $ n, \ell$ be non-negative integers. Also, let $\bm{\beta}=(\beta_1, \ldots, \beta_n)$ be an $n$-dimensional vector of non-negative integers. Then we call any ordered tuple of the form $(n,\ell, \bm{\beta})$ or $(n,\bm{\beta})$ a \emph{multi-index}. For any function $f:\R^d \times\cP_2(\R^d) \to \R$, the derivative $D^{(n,\ell, \bm{\beta})} f(x, \mu , v_1, \ldots, v_n)$  is defined as 
$$D^{(n,\ell, \bm{\beta})} f (x, \mu , v_1, \ldots, v_n) :=  \partial^{\beta_n}_{v_n} \ldots \partial^{\beta_1}_{v_1} \partial^{\ell}_x \pnmu f( x, \mu , v_1, \ldots, v_n),$$
if this derivative is well-defined.
For any function $\Phi: \mathcal{P}_2 ( \bR^d) \to \bR$, we define
$$D^{(n,\bm{\beta})} \Phi( \mu , v_1, \ldots, v_n) :=  \partial^{\beta_n}_{v_n} \ldots \partial^{\beta_1}_{v_1} \pnmu \Phi( \mu , v_1, \ldots, v_n), $$ 
if this derivative is well-defined. Finally, we also define the \emph{order} \footnote{ We do not consider `zeroth' order derivatives in our definition, i.e. at least one of $n$, $\beta_1, \ldots, \beta_n$ and $\ell$ must be non-zero, for every multi-index $\big(n, \ell, (\beta_1, \ldots, \beta_n) \big)$.} $ |(n,\ell, \bm{\beta})| $ (resp.  $|(n,\bm{\beta})|$ ) by
\begin{equation} |(n,\ell, \bm{\beta})|:= n+ \beta_1 + \ldots + \beta_n + \ell , \quad \quad |(n,\bm{\beta})|:= n+ \beta_1 + \ldots + \beta_n .  \label{eq:orderdef}
\end{equation}
\end{definition}
We now introduce a convenient class of functionals of measure that will serve as a hypothesis for some results.  

\begin{definition} \label{def Mk f} 
A function $f: \bR^d \times \cP_2 ( \bR^d) \to \bR$  belongs to class $\cM_k( \bR^d \times \cP_2 ( \bR^d)) $, if the derivatives  $D^{(n,\ell, \bm{\beta})} f (x,\mu,v_1, \ldots, v_n)$  exist
%and $D^{(n,\ell, \bm{\beta})} \sigma (x,\mu, v_1, \ldots, v_n)$ exist 
for every multi-index $(n,\ell, \bm{\beta})$ such that $|(n,\ell, \bm{\beta})| \leq k$ and satisfy
\begin{enumerate}[(i)]
\item \begin{equation}   
\quad \big| D^{(n,\ell, \bm{\beta})} f (x,\mu, v_1, \ldots, v_n) \big|  
 \leq  C, 
 %\quad \quad  \big| D^{(n,\ell, \bm{\beta})}  \sigma (x,\mu, v_1, \ldots, v_n) \big| \leq C, 
 \label{eq:boundf}
 \end{equation}
\item \begin{align}
\quad &\Big| D^{(n,\ell, \bm{\beta})} f (x,\mu, v_1, \ldots, v_n) -D^{(n,\ell, \bm{\beta})} f  (x', \mu' ,v'_1, \ldots, v'_n)  \Big| 
        \notag\\&\phantom{\Big| D^{(n,\ell, \bm{\beta})} f (x,\mu, v_1, \ldots, v_n) }\leq C \bigg( |x-x'| + \sum_{i=1}^n|v_i-v'_i|+ W_2 (\mu,\mu') \bigg),  \label{eq:Lipf}
\end{align} 
for any $x,x',v_1,v'_1,\ldots, v_n, v'_n \in \bR^d$ and $\mu, \mu' \in \mathcal{P}_2(\bR^d)$, for some constant $C>0$.
\end{enumerate}
Any function $f:\cP_2(\R^d) \to \R$ can be  extended to $\R^d \times \cP_2(\R^d)$ naturally by $(x,\mu)\mapsto f(\mu)$, for all $x \in \R^d$. This allows us to define the class $\cM_k(\cP_2 ( \bR^d)) $.
\end{definition}
\begin{remark}
By the mean-value theorem, assumption \eqref{eq:Lipf} automatically holds for any $|(n,\ell, \bm{\beta})| < k$, by assumption \eqref{eq:boundf}.
\end{remark}
For the time-dependent case, we extend the previous definition as follows.
\begin{definition} \label{def Mk time space measure} 
 A function $\cV: [0,T] \times \mathcal{P}_2 ( \bR^d) \to \bR$ is said  to be in $\cM_k ( [0,T] \times \mathcal{P}_2 ( \bR^d))$, if
\begin{enumerate}[(i)]
\item $s \mapsto \cV(s,\mu)$ is continuously differentiable on $[0,T]$.
\item $\cV(s, \cdot) \in \cM_k (\mathcal{P}_2 ( \bR^d))$, for each $s \in [0,T]$, where the constant $C$ in \eqref{eq:boundf} and \eqref{eq:Lipf} is uniform in $s \in [0,T]$.
\item All derivatives in measure (including the zeroth order derivative) of $\cV(\cdot,  \cdot)$ up to the $k$th order are jointly continuous in time and measure.
\end{enumerate}
\end{definition}
Examples regarding the computations of L-derivatives for various functionals of measures are given in Section 5.2.2 of \cite{carmona2017probabilistic}. In particular, Example 5.2.2.3 from \cite{carmona2017probabilistic} gives an analogue version to Theorem \ref{chain rule ii} for L-derivatives. The following examples are a direct consequence of this result. 

\begin{example}
The following functions $F: \bR^d \times \cP_2(\bR^d) \to \bR$ belong to $\cM_k(\bR^d \times \cP_2(\bR^d))$. 
\begin{enumerate}[(i)]
    \item $p$th-degree interaction:
    $$ F(x, \mu)= \int_{\bR^d} \ldots \int_{\bR^d} \varphi(x, y_1, \ldots, y_p) \, \mu(dy_1) \ldots \mu(dy_p),$$ 
    where $ \varphi :(\bR^d)^{p+1} \to \bR$ is  bounded and $C^k$ with bounded and Lipschitz partial derivatives up to and including order $k$. 
    \item $p$th-degree polynomial on the Wasserstein space:
    $$ F(x, \mu)= \prod_{i=1}^p \int_{\bR^d} \varphi_i (x, y) \, \mu(dy),$$
    where, for each $i \in \{1 , \ldots, p \}$,  $ \varphi_i:(\bR^d)^{2} \to \bR$ is  bounded and $C^k$ with bounded and Lipschitz partial derivatives up to and including order $k$. 
\end{enumerate}
\end{example}
The following results  establish links between linear functional derivatives and L-derivatives.
\begin{theorem}[Theorem 3.3.2 of \cite{tse2019quantitative}] \label{eq: L deriv linear relation higher}
Consider $U:\cP_2(\bR^d) \to \bR$. Suppose that $\partial^k_{\mu}U$ exists and is Lipschitz continuous. Then the $kth$ order linear functional derivative of $U$ exists and satisfies the relation
$$ \partial^k_{\mu}U(\mu,y_1, \ldots, y_k) = \partial_{y_1} \ldots \partial_{y_k} \frac{ \delta^k U}{\delta m^k} (\mu,y_1, \ldots, y_k). $$ \end{theorem} 
\begin{lemma}[Lemma 2.5 of \cite{chassagneux2019weak}] \label{mk sk,k} 
Let $k \geq 2$. Then $\cM_k( \cP_2 ( \bR^d)) \subseteq \cS_{k,k} ( \cP_2 ( \bR^d))$.
\end{lemma}
The next Lemma gives sufficient conditions in terms of L-derivatives for the hypotheses of Corollary \ref{cortlcs2} to be satisfied.
\begin{lemma} \label{M2 Lip cont} 
Let $U \in \cM_2(\cP_2(\bR^d))$. Then the second order linear functional derivative of $U$ satisfies
$$\exists C<\infty,\;\mu \in \cP_2(\bR^d),\;\forall x,y, \tilde{y} \in \bR^d,\;\bigg| \sld[U](\mu,x,y) - \sld[U](\mu,x,\tilde{y}) \bigg| \leq C |x| |y-\tilde{y}|,$$
Moreover, $U$ satisfies the hypotheses of Corollary \ref{cortlcs2} for each $\ell\ge 4$.
\end{lemma}
\begin{proof}
For any $C^{2}$ function $F: \bR^{2d} \to \bR$ let $\nabla_1F$ and $\nabla^2_{12} F$ denote the vector and the matrix with respective entries $\frac{\partial}{\partial_{z_i}}F(z_1,\hdots,z_d,z_{d+1},\hdots,z_{2d})$ and $\frac{\partial^2}{\partial_{z_i}\partial_{z_{d+j}}}F(z_1,\hdots,z_d,z_{d+1},\hdots,z_{2d})$, $1\le i,j\le d$. For points $x,\tilde x, y, \tilde y\in \bR^d$,
\begin{eqnarray}
&&  F(x, y) - F(\tilde x, y)- F(x,\tilde y) + F(\tilde x,\tilde y) \nonumber \\
& = & \int_0^1 (x-\tilde x) .\nabla_{1}F( (1-s)\tilde x+ sx,y)\,ds - \int_0^1 (x-\tilde x) .\nabla_1F( (1-s)\tilde x+ sx,\tilde y)\,ds  \nonumber \\
& = & \int_0^1 \int_0^1 (x-\tilde x) \cdot \nabla_{12}^2 F( (1-s)\tilde x+ sx, (1-t)\tilde y+ ty)(y-\tilde y) \, dt \, ds. \nonumber 
\end{eqnarray}
Therefore, by Theorem \ref{eq: L deriv linear relation higher},
\begin{eqnarray}
&&  \sld[U](\mu,x, y) - \sld[U](\mu,\tilde x, y)- \sld[U](\mu,x,\tilde y) + \sld[U](\mu,\tilde x,\tilde y) \nonumber \\
& = & \int_0^1 \int_0^1 (x-\tilde x) \cdot \ptwomu U(\mu, (1-s)\tilde x+ sx, (1-t)\tilde y+ ty)(y-\tilde y) \, dt \, ds. \nonumber
\end{eqnarray}
By setting $\tilde{x}:=0$, the normalisation condition \eqref{eq: normalisation linear functional deriatives} gives 
$$ \sld[U](\mu,x, y) - \sld[U](\mu,x,\tilde y) =  \int_0^1 \int_0^1 x \cdot \ptwomu U(\mu, sx, (1-t)\tilde y+ ty)(y-\tilde y) \, dt \, ds,$$ 
from which we conclude the result by the boundedness of $\ptwomu U$.

Let now $\ell\ge 2$. By Lemma \ref{mk sk,k}, $U\in\cS_{2,2} ( \cP_2 ( \bR^d))\subset\cS_{2,2} ( \cP_\ell ( \bR^d))$. Moreover, the first statement ensures that the H\"older continuity condition in Corollary \ref{cortlcs2} is satisfied for $\alpha=1$ (Lipschitz continuity). Last, when $\ell\ge 4$, $\cS_{2,2} ( \cP_\ell ( \bR^d))\subset\cS_{1,\ell/2}(\cP_\ell ( \bR^d))$ and all the hypotheses in this corollary are satisfied.
\end{proof}
\subsection{Mean-field fluctuation}
We define Lipschitz-continuous (w.r.t. the product topology of $\cP_2(\bR^d) \times \bR^d$) functions $b: \bR^d \times \mathcal{P}_2(\bR^d) \to \bR^d$ and $ \sigma : \bR^d \times \mathcal{P}_2(\bR^d) \to \bR^{d} \otimes \bR^{d'}$ as the drift and diffusion coefficients respectively. Let $(\Omega, \cF, \bP)$ be an atomless, complete probability space, on which we consider an interacting particle system
\begin{equation} \label{eq:particlesystem} \begin{cases} 
      Y^{i,N}_t = \xi_i + \int_0^t  b(Y^{i,N}_s, \nlaw[s]) \,ds + \int_0^t \sigma(Y^{i,N}_s, \nlaw[s]) \,dW^i_s, \quad 1 \leq i \leq N, \quad t \ge 0, \\
      \\
       \nlaw[s] := \frac{1}{N} \sum_{i=1}^N \delta_{Y^{i,N}_s}, 
   \end{cases}
\end{equation}
 where $W^i, 1 \leq i \leq N,$ are independent ${d'}$-dimensional Brownian motions and $\xi_i,  1 \leq i \leq N,$ are i.i.d. random variables  with law $\nu \in \cP_2(\bR^d)$ that are also independent of $W^1, \ldots, W^N$. 
 This type of equations provides a probabilistic representation to many high-dimensional PDEs arising from kinetic theory and mean-field games. 
A standard approximation of this particle system is through the mean-field limit of $\nlaw[t]$ (by the theory of propagation of chaos), which leads to the consideration of a corresponding McKean-Vlasov SDE given by
\begin{equation} \label{eq:MVSDE} \begin{cases} 
      X_t = \xi + \int_0^t b(X_s, \law[s]) \,ds + \int_0^t \sigma (X_s, \law[s]) \,dW_s, \quad \quad t\ge 0, \\
      \\
       \law[s] := \text{Law} (X_s), 
   \end{cases}
\end{equation}
where $W$ is a ${d'}$-dimensional Brownian motion and $\xi \sim \nu$ is independent of $W$.  Analyses of the approximation of  \eqref{eq:particlesystem} by the mean-field limiting equation \eqref{eq:MVSDE} are widely considered in the literature, such as  \cite{bossy1997stochastic},   \cite{meleard1996asymptotic} and \cite{sznitman1991topics}. In particular, by \cite{sznitman1991topics}, the condition of Lipschitz continuity of $b$ and $\sigma$ ensures existence and uniqueness of the solutions to \eqref{eq:particlesystem} and 
\eqref{eq:MVSDE} respectively. 

We consider the nonlinear fluctuation between the standard particle system \eqref{eq:particlesystem} and its standard McKean-Vlasov limiting equation \eqref{eq:MVSDE} under non-linear functionals $\Phi \in \cM_k( \cP_2(\bR^d) )$, i.e. we consider the limiting distribution of the process
$$  F^N:= \sqrt{N}	 \big[ \Phi(\nlaw[\cdot])  - \Phi ( \law[\cdot] ) \big]$$ 
in the space $C(\bR_{+}, \bR)$. 

The main analysis depends on the following function:
$\cV : \bR_+\times \mathcal{P}_2 (\bR^d) \to \bR $ defined by 
\begin{equation}  \cV( t, \rvlaw[\theta]) = \Phi \big( \rvlaw[{X^{\theta}_t}] \big) \label{eq: defofflow}
\end{equation}
where, for $\theta$ an $\bR^d$-valued random vector independent of $W$, $$X^{\theta}_t= \theta + \int_0^t b(X^{\theta}_s, \rvlaw[{X^{\theta}_s}]) \,ds + \int_0^t \sigma (X^{\theta}_s, \rvlaw[{X^{\theta}_s}]) \,dW_s,\quad \quad t\ge 0.$$
It is proven in Theorem 7.2 of  \cite{buckdahn2017mean} that, if $\nu \in \cP_2(\bR^d)$, $\Phi \in \cM_2( \cP_2(\bR^d) )$ and $b_i,\sigma_{i,j} \in \cM_2( \bR^d \times \cP_2(\bR^d) ),$ for $i \in \{1, \ldots, d \}$ and $j \in \{1, \ldots, {d'} \}$, then $\cV$ satisfies the \emph{master equation} given by
\begin{equation}  \label{eq pde measure} \begin{cases} 
       \partial_s \cV(s, \mu) =\int_{\bR^d}  \big[ \partial_{\mu} \cV (s, \mu) (x) \cdot b( x, \mu)  + \frac{1}{2} \text{Tr} \big( \partial_v \partial_{\mu} \cV  ( s,\mu) ( x) a( x, \mu) \big) \big] \, \mu (dx) , & s \ge 0, \\
      &  \\
      \cV(0,\mu) = \Phi ( \mu), &    \end{cases}
\end{equation} 
where 
\begin{equation} a(x, \mu):= \sigma(x,\mu) \sigma(x,\mu)^T. \label{eq def of a }  \end{equation} 
% For every $t \in [0,T]$, we define
% a function $\cV_t : [0,t] \times \mathcal{P}_2 (\bR^d) \to \bR $ by 
% \begin{equation}   \cV_t( s, \rvlaw[\theta]) = \Phi \big( \rvlaw[{X^{s, \theta}_t}] \big), \quad \quad s \in [0,t].  \label{eq:Vt} \end{equation}
% By equation $(6.12)$ in \cite{buckdahn2017mean}, it is shown that 
% \begin{equation} \cV_t(s,\mu) =  \cV(s+T-t, \mu). \label{eq:timeinvar} \end{equation}
By the initial condition of \eqref{eq: defofflow}, along with the definition of $\cV$, we have the decomposition
\begin{eqnarray}
\Phi(\nlaw[t])  - \Phi ( \law[t] ) & = &   \cV(0,\nlaw[t]) - \cV(t,\nu) \nonumber \\
% & = & \cV(T,\nlaw[t]) - \cV(T-t,\nu) \nonumber \\
& = & \big( \cV(0,\nlaw[t]) - \cV(t,\nlaw[0]) \big)+ \big( \cV(t,\nlaw[0]) - \cV(t,\nu) \big). \label{eq:errordecomp}
\end{eqnarray}
To treat the first term, we define a finite dimensional projection  $V:[0,t]\times(\bR^d)^N \to \bR$ by
\begin{equation}
    V(s  ,x_1,\ldots,x_N):= \cV \bigg( t-s, \frac{1}{N}\sum_{i=1}^N\delta_{x_i} \bigg). \label{eq:finitedimproj}
\end{equation} 
Then 
$$ \cV(0,\nlaw[t]) - \cV(t,\nlaw[0]) = V(t, Y^{1,N}_t, \ldots, Y^{N,N}_t) - V(0, Y^{1,N}_0, \ldots, Y^{N,N}_0).$$
We can now apply It\^{o}'s formula to this equality. Proposition 3.1 of \cite{chassagneux2014probabilistic} allows us to conclude that  $V$ is differentiable in the time component and twice-differentiable in the space components. Moreover, Proposition 3.1 of \cite{chassagneux2014probabilistic} expresses the first and second order partial derivatives of $V$ in terms of the L-derivatives of $\cV$. This allows us to use \eqref{eq pde measure} to obtain a cancellation in the L-derivatives (except the second order term). 

We now present the details of the above discussion (found in the proof of Theorem B.2 in \cite{szpruch2019antithetic}) as follows. Setting $\mathbf{Y}^N = (Y^{1,N}, Y^{2,N}, \ldots, Y^{N,N})$, we have

\begin{eqnarray}
&&  {\cV}(0,\nlaw[t]) - {\cV}(t,\nlaw[0])   \nonumber
  \\
& = & \bigg[ \int_0^{t} \frac{ \partial V}{ \partial s} (s, \mathbf{Y}^N_s) + \sum_{i=1}^N \frac{\partial V}{\partial x_i} ( s, \mathbf{Y}^N_s) \cdot  b \big( Y^{i,N}_s, \nlaw[s] \big)  
+ \frac{1}{2}  \text{Tr} \bigg( a \big( Y^{i,N}_s, \nlaw[s]  \big)   \sum_{i=1}^N \frac{\partial^2 V}{\partial x^2_i}  (s,\mathbf{Y}^N_s) \bigg) \,ds \bigg]   \nonumber \\
&& +  \sum_{i=1}^N \int_0^{t}  \sigma( Y^{i,N}_s, \nlaw[s] )^T \frac{\partial V}{\partial x_i} ( s,\mathbf{Y}^N_s) \cdot dW_s^i  \nonumber \\
& = &  \int_0^{t}  \partial_s {\cV} \big( s, \nlaw[s]  \big) + \sum_{i=1}^N \Bigg[ \frac{1}{N} \partial_{\mu} {\cV} \big( s, \nlaw[s]  \big)( {Y^{i,N}_s} ) \cdot b \big( Y^{i,N}_s, \nlaw[s]  \big)  \nonumber  \\
&& \, \, + \frac{1}{2}  \text{Tr} \Bigg( a \big( Y^{i,N}_s, \nlaw[s]  \big) \bigg( \frac{1}{N} \partial_{v}   \partial_{\mu} {\cV} \big( s, \nlaw[s]  \big)  (Y^{i,N}_s)  + \frac{1}{N^2} \partial^2_{\mu}  {\cV} \big( s, \nlaw[s] \big)  (Y^{i,N}_s, Y^{i,N}_s) \bigg) \Bigg) \Bigg] \,ds  \nonumber \\
&& +  \frac{1}{N} \sum_{i=1}^N \int_0^{t}  \sigma( Y^{i,N}_s, \nlaw[s] )^T \partial_{\mu} {\cV} \big( s, \nlaw[s]  \big) (  Y^{i,N}_s) \cdot dW_s^i. \label{eq: eqn Ito} 
\end{eqnarray}
By \eqref{eq:errordecomp}, \eqref{eq: eqn Ito}  and PDE \eqref{eq pde measure} evaluated at  $(s,\nlaw[s])_{s \in [0,t]}$, the expression simplifies to
\begin{eqnarray} 
	\Phi(\nlaw[t])  - \Phi ( \law[t] )  & =  & \big( \cV(t,\nlaw[0]) - \cV(t,\nu) \big)  \nonumber \\
	 && + \int_0^{t} \frac{1}{2} \Bigg[\frac{1}{N^2} \sum_{i=1}^N  \text{Tr} \bigg( a \big(Y^{i,N}_s, \nlaw[s]  \big)    \partial^2_{\mu}  {\cV} \big( s, \nlaw[s]  \big)  (Y^{i,N}_s, Y^{i,N}_s) \bigg) \Bigg] \,ds \nonumber  \\
&& +\frac{1}{N}\sum_{i=1}^N   \int_0^{t}  \sigma( Y^{i,N}_s, \nlaw[s] )^T \partial_{\mu} {\cV} \big( s, \nlaw[s]  \big) ( Y^{i,N}_s) \cdot dW_s^i.  \nonumber
\end{eqnarray}

The following proposition states this result rigorously. 
\begin{proposition} \label{thm reg} 
Let $k \geq 2$. Suppose that $\nu \in \cP_2(\bR^d)$, $\Phi \in \cM_k( \cP_2(\bR^d) )$ and $b_i,\sigma_{i,j} \in \cM_k( \bR^d \times \cP_2(\bR^d) ),$ for $i \in \{1, \ldots, d \}$ and that $j \in \{1, \ldots, {d'} \}$. Then, for each $T>0$, $ \cV \in \cM_k([0,T] \times \cP_2(\bR^d))$ and the marginal fluctuation at time $t \in [0,T]$ can be expressed as \begin{eqnarray}
 \sqrt{N}	 \Big[ \Phi(\nlaw[t])  - \Phi ( \law[t] ) \Big] &  =   &  \sqrt{N} \big( \cV(t,\nlaw[0]) - \cV(t,\nu) \big)  \nonumber \\ 
 & & + \int_0^{t} \frac{1}{2} \Bigg[\frac{1}{N^{3/2}} \sum_{i=1}^N  \text{Tr} \bigg( a \big(Y^{i,N}_s, \nlaw[s]  \big)    \partial^2_{\mu}  {\cV} \big( t-s, \nlaw[s]  \big)  (Y^{i,N}_s, Y^{i,N}_s) \bigg) \Bigg] \,ds \nonumber  \\
&& +\frac{1}{\sqrt{N}}\sum_{i=1}^N   \int_0^{t}  \sigma( Y^{i,N}_s, \nlaw[s] )^T \partial_{\mu} {\cV} \big( t-s, \nlaw[s]  \big) ( Y^{i,N}_s) \cdot dW_s^i. \label{eq:fluct} 
\end{eqnarray} 
\end{proposition}
\begin{proof}
The statement concerning the regularity of $\cV$ comes from Theorem 2.15 of \cite{chassagneux2019weak} (see also Theorem 7.2 in \cite{buckdahn2017mean} for the special case $k=2$ and \cite{tse2021higher} for a related proof from the perspective of PDE analysis).  Equation \eqref{eq:fluct}  comes from (B.7) of \cite{szpruch2019antithetic}.
\end{proof}
{
\begin{lemma} \label{lemma four moment of time diff} 
Suppose that $\Phi \in \cM_5( \cP_2(\bR^d) )$ and that $b_i,\sigma_{i,j} \in \cM_5( \bR^d \times \cP_2(\bR^d) ),$ for $i \in \{1, \ldots, d \}$ and $j \in \{1, \ldots, {d'} \}$. Suppose that $b$ and $\sigma$ are uniformly bounded.  Let $\nu \in \cP_{12}(\bR^d)$. Then the function $\cV$ $($defined by \eqref{eq: defofflow}$)$ satisfies
$$ \bE \Big| \big( \cV(t_2, \mu^N_0) - \cV(t_2, \nu) \big) - \big( \cV(t_1, \mu^N_0) - \cV(t_1, \nu) \big) \Big|^4 \leq C \frac{|t_1 - t_2|^4}{N^2},$$ 
for every $t_1, t_2 \in [0,T],$ for some $C>0$. 
\end{lemma}

\begin{proof}
For simplicity of notations, the proof is presented in dimension 1.
By \eqref{eq pde measure}, $\partial_t \cV$ exists and is given by
$$ \partial_t \cV(t, \mu) =\int_{\bR}  \bigg[ \partial_{\mu} \cV (t, \mu) (x)  b( x, \mu)  + \frac{1}{2}   \partial_x \partial_{\mu} \cV  ( t,\mu) ( x) a( x, \mu)  \bigg] \, \mu (dx). $$
By Proposition \ref{thm reg}, $\cV \in \cM_5([0,T] \times \cP_2(\bR^d))$. By part (ii) of Theorem 3.2.3 of \cite{tse2019quantitative}, 
\begin{eqnarray*}
&& \pmu \bigg[ \partial_{\mu} \cV (t, \mu) (x)  b( x, \mu)  + \frac{1}{2}   \partial_x \partial_{\mu} \cV  ( t,\mu) ( x) a( x, \mu) \bigg] (y) \nonumber \\
& = & \partial^2_{\mu} \cV (t, \mu) (x,y) b(x, \mu) +  \partial_{\mu} \cV (t, \mu) (x) \pmu b( x, \mu)(y) \nonumber \\
&& + \, \frac{1}{2} \partial_x \ptwomu \cV(t, \mu)(x,y) a(x,\mu) + \frac{1}{2}   \partial_x \partial_{\mu} \cV  ( t,\mu) ( x) \pmu a( x, \mu) (y). 
\end{eqnarray*}
By the hypotheses of Lemma \ref{lemma four moment of time diff}, we can apply Example 3 of Section 5.2.2 of \cite{carmona2017probabilistic} to yield
\begin{eqnarray*}
\pmu \big[ \partial_t \cV(t, \mu) \big] (y) & = & \partial_y \Big[ \partial_{\mu} \cV (t, \mu) (y)  b( y, \mu)  + \frac{1}{2}   \partial_y \partial_{\mu} \cV  ( t,\mu) ( y) a( y, \mu) \Big] \nonumber \\
&& + \int_{\bR} \bigg[ \partial^2_{\mu} \cV (t, \mu) (x,y) b(x, \mu) +  \partial_{\mu} \cV (t, \mu) (x) \pmu b( x, \mu)(y) \nonumber \\
&& + \, \frac{1}{2} \partial_x \ptwomu \cV(t, \mu)(x,y) a(x,\mu) + \frac{1}{2}   \partial_x \partial_{\mu} \cV  ( t,\mu) ( x) \pmu a( x, \mu) (y) \bigg] \, \mu(dx). 
\end{eqnarray*}
One can easily check that
$$ \sup_{t \in [0,T], \, \mu \in \cP_2(\bR^d), \, y \in \bR^d} \Big| \pmu \big[ \partial_t \cV(t, \mu) \big] (y) \Big| < +\infty  $$ 
and
$$ \sup_{t \in [0,T]} \Big| \pmu \big[ \partial_t \cV(t, \mu_1) \big] (y_1)  - \pmu \big[\partial_t \cV(t, \mu_2) \big] (y_2)  \Big| \leq C \big( |y_1 - y_2| + W_2(\mu_1, \mu_2) \big),$$ 
for some finite constant $C$, with the domination of the integral with respect to $\mu_1-\mu_2$ of the function of $x$ (with other arguments frozen in $\mu_1$ and $y_1$) coming from Lipschitz continuity, Kantorovitch-Rubinstein duality and the inequality $W_1\le W_2$.

Iterating this argument for higher order derivatives of $\partial_t \cV$ up to order 3,
we deduce that $\partial_t \cV \in \cM_3([0,T] \times \cP_2(\bR^d))$. 

Lemma 3.2 in \cite{szpruch2019antithetic} states that, for any function $f \in \cM_3(\cP_2(\bR^d))$, measure $m_0 \in \cP_{12}(\bR^d)$ and $m^N= \frac{1}{N} \sum_{i=1}^N \delta_{{\zeta}_i}$, where ${\zeta}_1, \ldots, {\zeta}_N$ are i.i.d samples with law $m_0$, there exists an absolute constant $C>0$ (which does not depend on $f$, ${\zeta}_1, \ldots, {\zeta}_N$ and $m_0$) such that  
\begin{equation} \bE \big[ \big| f( m^N)- f(m_0) \big|^4 \big] \leq \frac{C}{N^2} \prod_{i=1}^3 \Big( 1+ \| \partial^i_{\mu} f \|^4_{\infty}  \Big) \bigg( 1+ \int_{\bR^d} |x|^{12} \, m_0 (dx) \bigg) . \label{cite 4th moment bound} \end{equation}

Take any $t_1, t_2 \in [0,T]$ such that $t_1 < t_2$. By \eqref{cite 4th moment bound} and H\"{o}lder's inequality,
there exists an absolute constant $C>0$ (which does not depend on $\cV$, $\mu^N_0$ and $\nu$) such that
\begin{eqnarray}
&& \bE \Big| \big( \cV(t_2, \mu^N_0) - \cV(t_2, \nu) \big) - \big( \cV(t_1, \mu^N_0) - \cV(t_1, \nu) \big) \Big|^4 \nonumber \\
& = & \bE \bigg| \int_{t_1}^{t_2} \big[ \partial_t \cV(t, \mu^N_0) - \partial_t \cV(t, \nu) \big] \,dt \bigg|^4 \nonumber \\
& \leq & |t_2 - t_1|^3 \int_{t_1}^{t_2} \bE \big| \partial_t \cV(t, \mu^N_0) - \partial_t \cV(t, \nu) \big|^4 \,dt \nonumber \\
& \leq & |t_2 - t_1|^4 \bigg[ \frac{C}{N^2} \prod_{i=1}^3 \Big( 1+ \sup_{t \in [0,T]} \| \partial^i_{\mu} (\partial_t \cV(t, \cdot)) \|^4_{\infty}  \Big) \bigg( 1+ \int_{\bR^d} |x|^{12} \, \nu (dx) \bigg) \bigg] . \nonumber 
\end{eqnarray}
\end{proof}
}
The following theorem concerns the limiting distribution of $F^N$.
\begin{theorem}
Suppose that $\Phi \in \cM_{{5}}( \cP_2(\bR^d) )$ and that $b_i,\sigma_{i,j} \in \cM_{5}( \bR^d \times \cP_2(\bR^d) ),$ for $i \in \{1, \ldots, d \}$ and $j \in \{1, \ldots, {d'} \}$. Let $\nu \in \cP_{12}(\bR^d)$.  {Moreover, suppose that one of the following two conditions is satisfied:
\begin{equation*} \begin{cases} 
     \nu \text{ is a Dirac mass, i.e. } \nu = \delta_c, \text{ for some } c \in \bR^d, \\
       &\\
      b \text{ and } \sigma \text{ are uniformly bounded.} 
   \end{cases} 
\end{equation*} }
Then, in $C(\bR_+, \bR)$, the process  $$F^N:= \sqrt{N}	 \big[ \Phi(\nlaw[\cdot])  - \Phi ( \law[\cdot] ) \big]$$ converges weakly to  a Gaussian process $L$ whose finite dimensional distribution $(L_{t_1}, \ldots, L_{t_K})$, $0 \leq t_1 \leq \ldots \leq t_K $, has a zero expectation vector and covariance matrix $\Sigma$ given by 
\begin{eqnarray}\Sigma_{i,j} & := & \text{\emph{Cov}} \bigg( \frac{\delta \cV}{\delta m} (t_i,\nu, \xi_1), \frac{\delta \cV}{\delta m} (t_j,\nu, \xi_1)  \bigg) \nonumber \\
&& + \bE \bigg[ \int_0^{t_i \wedge t_j}  \partial_{\mu} {\cV} \big( t_i-s, \law[s]  \big) (X_s)^T  a(X_s, \law[s])\partial_{\mu} {\cV} \big( t_j-s, \law[s]  \big) (X_s) \, ds  \bigg]. \end{eqnarray}
\end{theorem}
\begin{proof}
Firstly, by \eqref{eq:fluct}, we decompose $F^N$ as 
$$ F^N_t = \Theta^N_t + \Lambda^N_t,$$
where
$$ \Theta^N_t:=  \int_0^{t} \frac{1}{2} \Bigg[\frac{1}{N^{3/2}} \sum_{i=1}^N  \text{Tr} \bigg( a \big(Y^{i,N}_s, \nlaw[s]  \big)    \partial^2_{\mu}  {\cV} \big( t-s, \nlaw[s]  \big)  (Y^{i,N}_s, Y^{i,N}_s) \bigg) \Bigg] \,ds $$ 
and
$$ \Lambda^N_t:= \sqrt{N} \big( \cV(t,\nlaw[0]) - \cV(t,\nu) \big)   +\frac{1}{\sqrt{N}}\sum_{i=1}^N   \int_0^{t}  \sigma( Y^{i,N}_s, \nlaw[s] )^T \partial_{\mu} {\cV} \big( t-s, \nlaw[s]  \big) ( Y^{i,N}_s) \cdot dW_s^i.$$
By Lemma \ref{mk sk,k} and Corollary \ref{cortlcs2} applied with $\ell=12$, 
\begin{equation}  \bE \Big|\Lambda^N_0 \Big|  = \bE \Big|\sqrt{N} \big( \cV(0,\nlaw[0]) - \cV(0,\nu) \big) \Big|  \leq C, \label{tightness 1} 
\end{equation}
for some constant $C$ that does not depend on $N$. Since $b$ and $\sigma$ are Lipschitz (w.r.t. the Euclidean and $W_2$ norms respectively) and $\nu \in \cP_{12}(\bR^d)$, we have 
\begin{equation} \label{eq:boundedinL12}
     \sup_{u \in [0,t]} \bE[|X_u|^{12}]< +\infty \quad \quad \text{and} \quad \quad \sup_{N \in \bN} \sup_{u \in [0,t]} \bE \bigg[ \frac{1}{N} \sum_{i=1}^N |Y^{i,N}_{u}|^{12} \bigg] < + \infty,
\end{equation}
for any $t>0$. Consequently, by \eqref{eq:boundedinL12} and the fact that  $ \cV \in \cM_4([0,T] \times \cP_2(\bR^d))$ (which implies boundedness of $\ptwomu \cV$ by definition) for any $T>0$, we deduce that, for any $t>0$,
\begin{eqnarray}
\bE| \Theta^N_t|^2 & = & \bE \Bigg[ \bigg| \int_0^{t} \frac{1}{2N^{3/2}} \sum_{i=1}^N  \text{Tr} \bigg( a \big(Y^{i,N}_s, \nlaw[s]  \big)    \partial^2_{\mu}  {\cV} \big( t-s, \nlaw[s]  \big)  (Y^{i,N}_s, Y^{i,N}_s) \bigg)  \,ds \bigg|^2 \Bigg] \nonumber  \\
& \leq & t \bE \int_0^{t} \bigg| \frac{1}{2N^{3/2}} \sum_{i=1}^N  \text{Tr} \bigg( a \big(Y^{i,N}_s, \nlaw[s]  \big)    \partial^2_{\mu}  {\cV} \big( t-s, \nlaw[s]  \big)  (Y^{i,N}_s, Y^{i,N}_s) \bigg) \bigg|^2  \,ds \xrightarrow{N \to \infty} 0. \label{L2 limit theta N}
\end{eqnarray} 
It follows by a similar argument that for any $t_1, t_2 \in [0,T]$, there exists a constant $C_T>0 $ such that
\begin{equation}  \bE| \Theta^N_{t_2} - \Theta^N_{t_1} |^4 \leq C_T|t_2-t_1|^4.  \label{tightness 3} \end{equation}
{
Let 
$$ \cI^N_t :=  \frac{1}{\sqrt{N}}\sum_{i=1}^N   \int_0^{t}  \sigma( Y^{i,N}_s, \nlaw[s] )^T \partial_{\mu} {\cV} \big( t-s, \nlaw[s]  \big) ( Y^{i,N}_s) \cdot dW_s^i. $$
By Lemma 6.1 in \cite{buckdahn2017mean}, 
\begin{equation} 
\forall t_1,t_2\in[0,T],\;\sup_{\mu \in \cP_2(\bR^d), \, y \in \bR^d} \big| \pmu \cV(t_1, \mu)(y) - \pmu \cV(t_2, \mu)(y) \big|    \leq  C_T |t_1 - t_2|^{1/2}, \label{time diff}
\end{equation} 
for some constant $C_T>0$ that only depends on $T$. 
Take any $t_1, t_2 \in [0,T]$ with $ t_1 < t_2$. Then, by \eqref{time diff} and the H\"{o}lder's inequality,
\begin{eqnarray}
\bE   \Big|\cI^N_{t_2}- \cI^N_{t_1} \Big|^4
& \leq & 8 \bE  \bigg[ \bigg( \frac{1}{\sqrt{N}}\sum_{i=1}^N   \int_{0}^{{t_1}}  \sigma( Y^{i,N}_s, \nlaw[s] )^T \partial_{\mu} {\cV} \big( t_2 -s, \nlaw[s]  \big) ( Y^{i,N}_s) \nonumber \\
&& -  \sigma( Y^{i,N}_s, \nlaw[s] )^T \partial_{\mu} {\cV} \big( t_1 -s, \nlaw[s]  \big) ( Y^{i,N}_s) \cdot dW_s^i \bigg)^4  \bigg] \nonumber \\
&& + 8 \bE  \bigg[ \bigg( \frac{1}{\sqrt{N}}\sum_{i=1}^N   \int_{t_1}^{{t_2}}  \sigma( Y^{i,N}_s, \nlaw[s] )^T \partial_{\mu} {\cV} \big( t_2-s, \nlaw[s]  \big) ( Y^{i,N}_s) \cdot dW_s^i \bigg)^4 \bigg]. \label{split tightness} 
\end{eqnarray}
By the Burkholder-Davis-Gundy, Jensen's and H\"{o}lder's inequalities, the second term of \eqref{split tightness} can be bounded by
\begin{eqnarray}
&& \bE  \bigg[ \bigg( \frac{1}{\sqrt{N}}\sum_{i=1}^N   \int_{t_1}^{{t_2}}  \sigma( Y^{i,N}_s, \nlaw[s] )^T \partial_{\mu} {\cV} \big( t_2-s, \nlaw[s]  \big) ( Y^{i,N}_s) \cdot dW_s^i \bigg)^4 \bigg] \nonumber \\
& \leq  & C^{(1)}_T \bE \bigg[ \lev \frac{1}{\sqrt{N}}\sum_{i=1}^N  \int_{t_1}^{\cdot}\sigma( Y^{i,N}_s, \nlaw[s] )^T \partial_{\mu} {\cV} \big( t_2-s, \nlaw[s]  \big) ( Y^{i,N}_s) \cdot dW_s^i \rev_{t_2}^{2} \bigg]  \nonumber  \\
& = & C^{(1)}_T \bE \bigg[ \bigg( \frac{1}{N} \sum_{i=1}^N \int_{t_1}^{t_2}  \Big|\sigma( Y^{i,N}_s, \nlaw[s] )^T \partial_{\mu} {\cV} \big( t_2-s, \nlaw[s]  \big) ( Y^{i,N}_s) \Big|^2 \,ds \bigg)^2 \bigg] \nonumber \\
& \leq & C^{(1)}_T \bE \bigg[  \frac{1}{N} \sum_{i=1}^N  \bigg( \int_{t_1}^{t_2}  \Big|\sigma( Y^{i,N}_s, \nlaw[s] )^T \partial_{\mu} {\cV} \big( t_2-s, \nlaw[s]  \big) ( Y^{i,N}_s) \Big|^2 \,ds \bigg)^2 \bigg] \nonumber \\
& \leq & C^{(2)}_T |t_2-t_1|^{2}, \nonumber 
\end{eqnarray}
for some constants $C^{(1)}_T, C^{(2)}_T$ that only depend on $T$. Repeating the same argument to the first term in \eqref{split tightness}, we observe by \eqref{time diff} that
\[
    \bE   \Big|\cI^N_{t_2}- \cI^N_{t_1} \Big|^4 \leq C_T |t_2 - t_1|^2.
\]
This estimate, alone when $\nu$ is a Dirac mass so that $\mu^N_0=\nu$ and $\Lambda^N_t=\cI^N_t$, and combined with Lemma \ref{lemma four moment of time diff} otherwise,  yields
\begin{equation}
    \bE   \Big|\Lambda^N_{t_2}- \Lambda^N_{t_1} \Big|^4 \leq C_T |t_2 - t_1|^2. \label{tightness 2} 
\end{equation}
}
 By \eqref{tightness 1},  \eqref{tightness 3} and \eqref{tightness 2}, we conclude that the sequence of probability measures $\{ \cL( F^N) \}_N$ is tight on $C(\bR_{+}, \bR)$ (see  Problem 2.4.11 in \cite{karatzas2012brownian}).

Next, we compute the weak limit of the finite dimensional distributions of $F^N$. We first define the coupling of \eqref{eq:MVSDE} given by
$$ X^i_t = \xi_i + \int_0^t b(X^i_s, \law[s])\, ds + \int_0^t \sigma(X^i_s, \law[s]) \, dW^i_s, \quad \quad t \in [0,T], \quad i \in \bN.$$ 
Let 
\begin{eqnarray}
E^N_t & := &  \frac{1}{\sqrt{N}}\sum_{i=1}^N   \int_0^{t}  \sigma( Y^{i,N}_s, \nlaw[s] )^T \partial_{\mu} {\cV} \big( t-s, \nlaw[s]  \big) ( Y^{i,N}_s) \cdot dW_s^i \nonumber \\
&& - \frac{1}{\sqrt{N}}\sum_{i=1}^N   \int_0^{t}  \sigma( X^{i}_s, \law[s] )^T \partial_{\mu} {\cV} \big( t-s, \law[s]  \big) ( X^{i}_s) \cdot dW_s^i , \nonumber
\end{eqnarray}
which implies that
\begin{eqnarray}
\bE| E^N_t|^2 &  = &  \frac{1}{{N}}\sum_{i=1}^N \bE \bigg[   \int_0^{t}  \bigg| \sigma( Y^{i,N}_s, \nlaw[s] )^T \partial_{\mu} {\cV} \big( t-s, \nlaw[s]  \big) ( Y^{i,N}_s)  \nonumber \\
&& -   \sigma( X^{i}_s, \law[s] )^T \partial_{\mu} {\cV} \big( t-s, \law[s]  \big) ( X^{i}_s) \bigg|^2 \,ds \bigg]. \label{eqn particles mckean diff} 
\end{eqnarray}
The assumptions that $\Phi \in \cM_4( \cP_2(\bR^d) )$ and that $b_i,\sigma_{i,j} \in \cM_4( \bR^d \times \cP_2(\bR^d) ),$ for $i \in \{1, \ldots, d \}$ and $j \in \{1, \ldots, {d'} \}$, allow us to repeat the calculations of  Theorem 5.1 in \cite{szpruch2019antithetic} to  deduce  \footnote{This is the main step which requires such a strong regularity assumption on $\nu$, $b$, $\sigma$ and $\Phi$, i.e. the assumption that $\nu \in \cP_{12}(\bR^d)$, $\Phi \in \cM_4( \cP_2(\bR^d) )$ and that $b_i,\sigma_{i,j} \in \cM_4( \bR^d \times \cP_2(\bR^d) ),$ for $i \in \{1, \ldots, d \}$ and $j \in \{1, \ldots, {d'} \}$. The reader is recommended to consult the proof of Theorem 5.1 in \cite{szpruch2019antithetic} for details.} that $\bE| E^N_t|^2 \to 0$, which implies that $E^N_t$ converges to $0$ in probability.

Let $0 \leq t_1 \leq \ldots \leq t_K$. Then $(E^N_{t_1}, E^N_{t_2}, \ldots, E^N_{t_K})$ converges in probability to $(0,0, \ldots, 0)$ and hence converges in distribution to $(0,0, \ldots, 0)$ as well. Similarly, by 
\eqref{L2 limit theta N}, $(\Theta^N_{t_1}, \Theta^N_{t_2}, \ldots, \Theta^N_{t_K})$ converges in distribution to $(0,0, \ldots, 0)$.

For simplicity of notations, for $0 \leq s \leq t $, we denote
$$  \Sigma ((s,t), x, \mu):=  \sigma( x, \mu )^T \partial_{\mu} {\cV} \big( t-s, \mu  \big) ( x) \in \bR^{d'}.$$ 
 Let $\theta_k$ be arbitrary real numbers, $k \in \{1, \ldots , K\}$. Then
 \begin{eqnarray}
 && \lim_{N \to \infty} \bE \bigg[ \exp \bigg\{ i \sum_{k=1}^K \theta_k F^N_{t_k } \bigg\} \bigg]  \nonumber \\
 & = & \lim_{N \to \infty} \Bigg[  \bE \bigg[ \exp \bigg\{ i \sqrt{N} \bigg( \sum_{k=1}^K \theta_k {\cV}( t_k, \nlaw[0]) -\sum_{k=1}^K \theta_k {\cV}(t_k, \nu) \bigg) \bigg\} \bigg] \nonumber \\
 && \,\, \times \bE \bigg[ \exp \bigg\{ i \sum_{k=1}^K \theta_k \bigg[ \frac{1}{\sqrt{N}} \sum_{j=1}^N  \int_0^{t_k} {\Sigma}((s,t_k),X^j_s, \law[s]) \cdot\,dW^j_s \bigg] \bigg\} \bigg] \Bigg] \nonumber \\
 & = &  \lim_{N \to \infty} \Bigg[  \bE \bigg[ \exp \bigg\{ i \sqrt{N} \bigg( \sum_{k=1}^K \theta_k {\cV}( t_k, \nlaw[0]) -\sum_{k=1}^K \theta_k {\cV}(t_k, \nu) \bigg) \bigg\} \bigg] \nonumber \\
 && \,\, \times \bE \bigg[ \exp \bigg\{ i \frac{1}{\sqrt{N}}  \sum_{j=1}^N  \bigg[ \sum_{k=1}^K   \theta_k \int_0^{t_k} {\Sigma}((s,t_k),X^j_s, \law[s]) \cdot \,dW^j_s \bigg] \bigg\} \bigg] \Bigg] \nonumber \\
 &= & \bE[\exp\{ iZ_1\}] \bE[ \exp\{ iZ_2\}],  \label{limiting calcul} 
 \end{eqnarray}
 where $Z_1$ and $Z_2$ are independent normal random variables given by
$$ Z_1 \sim N \bigg( 0, \text{Var}  \bigg( \sum_{k=1}^K \theta_k \frac{\delta {\cV}}{\delta m} (t_k,\nu, \xi_1) \bigg) \bigg)$$  
and
$$ Z_2 \sim N \bigg( 0, \bE \bigg[ \bigg(  \sum_{k=1}^K   \theta_k \int_0^{t_k} {\Sigma}((s,t_k),X^1_s, \law[s]) \cdot \,dW^1_s \bigg)^2 \bigg] \bigg), $$ 
by Corollary \ref{cortlcs2} along with  Lemma \ref{mk sk,k} and the classical central limit theorem respectively. Note that we can also rewrite the variances as
\begin{eqnarray}
&& \text{Var}  \bigg( \sum_{k=1}^K \theta_k \frac{\delta {\cV}}{\delta m} (t_k,\nu, \xi_1) \bigg) \nonumber \\
& = & \sum_{i,j=1}^K   \theta_i \theta_j \text{Cov} \bigg( \frac{\delta {\cV}}{\delta m} (t_i,\nu, \xi_1), \frac{\delta {\cV}}{\delta m} (t_j,\nu, \xi_1)  \bigg) \nonumber 
\end{eqnarray}
and
\begin{eqnarray}
&& \bE \bigg[ \bigg(  \sum_{k=1}^K  \theta_k \int_0^{t_k} {\Sigma}((s,t_k),X^1_s, \law[s]) \cdot \,dW^1_s \bigg)^2 \bigg] \nonumber \\
& = &  \sum_{i,j=1}^K  \theta_i \theta_j  \bE \bigg[ \bigg( \int_0^{t_i} {\Sigma}((s,t_i),X^1_s, \law[s]) \cdot \,dW^1_s \bigg)  \bigg( \int_0^{t_j} {\Sigma}((s,t_j),X^1_s, \law[s]) \cdot \,dW^1_s \bigg) \bigg] \nonumber \\
& = &  \sum_{i,j=1}^K   \theta_i \theta_j  \bE \bigg[ \int_0^{t_i \wedge t_j} {\Sigma}((s,t_i),X^1_s, \law[s]) \cdot {\Sigma}((s,t_j),X^1_s, \law[s]) \, ds  \bigg]. \nonumber 
\end{eqnarray}
By \eqref{limiting calcul}, 
\begin{eqnarray}
&& \lim_{N \to \infty} \bE \bigg[ \exp \bigg\{ i \sum_{k=1}^K \theta_k F^N_{t_k } \bigg\} \bigg]  \nonumber \\
& = & \exp \bigg\{ -\frac{1}{2} \sum_{i,j=1}^K   \theta_i \theta_j \bigg[ \text{Cov} \bigg( \frac{\delta {\cV}}{\delta m} (t_i,\nu, \xi_1), \frac{\delta {\cV}}{\delta m} (t_j,\nu, \xi_1)  \bigg) + \nonumber \\
&& \quad \bE \bigg[ \int_0^{t_i \wedge t_j} {\Sigma}((s,t_i),X^1_s, \law[s]) \cdot {\Sigma}((s,t_j),X^1_s, \law[s]) \, ds  \bigg] \bigg] \bigg\}. \nonumber 
\end{eqnarray}
By the L\'{e}vy's continuity theorem, this shows that the random vector $(\Lambda^N_{t_1}, \ldots, \Lambda^N_{t_K})$ converges weakly to some normal random vector $(L_{t_1}, \ldots, L_{t_K})$, whose expectation vector is zero and covariance matrix $\Sigma$ is given by
$$ \Sigma_{i,j} :=  \text{Cov} \bigg( \frac{\delta {\cV}}{\delta m} (t_i,\nu, \xi_1), \frac{\delta {\cV}}{\delta m} (t_j,\nu, \xi_1)  \bigg) + \bE \bigg[ \int_0^{t_i \wedge t_j} {\Sigma}((s,t_i),X^1_s, \law[s]) \cdot {\Sigma}((s,t_j),X^1_s, \law[s]) \, ds  \bigg]. $$ 
\end{proof}
\section{Appendix}
\begin{lemma}
  For $\ell\in(0,1)$, $W_\ell$ is a metric on ${\cP}_\ell(\R^d)$. Moreover, if $\mu\in\cP_\ell(\bR^d)$ and $(\mu_n)_{n\in\N}$ is a sequence in this space, then $\lim_{n\to\infty}W_\ell(\mu_n,\mu)=0$ iff $\mu_n$ converges weakly to $\mu$ as $n\to\infty$ and $\lim_{n\to\infty}\int_{\R^d}|x|^\ell\mu_n(dx)=\int_{\R^d}|x|^\ell\mu(dx)$.\label{lemwl}
\end{lemma}
\begin{proof}Let $\mu,\nu,\eta\in{\cP}_\ell(\R^d)$. Clearly $W_\ell(\mu,\nu)=W_\ell(\nu,\mu)$. By the triangle inequality and the subadditivity of $\R_+\ni u\mapsto u^\ell$, \begin{equation}
   \forall (x,y)\in\R^d\times\R^d, |x-y|^\ell\le|x|^\ell+|y|^\ell\mbox{ and }\left||x|^\ell-|y|^\ell\right|\le |x-y|^\ell.\label{triangsub}
 \end{equation} With the definition \eqref{defwl} of $W_\ell$, the first inequality implies that $W_\ell(\mu,\nu)\le\int_{\R^d}|x|^\ell\mu(dx)+\int_{\R^d}|y|^\ell\nu(dy)<\infty$. By Theorem 4.1 in \cite{villani2008optimal}, there is an optimal coupling $\rho$ between $\mu$ and $\nu$ i.e. an element of ${\cP}_\ell(\R^d\times\R^d)$ with first marginal $\mu$ and second marginal $\nu$ such that $W_\ell(\mu,\nu)=\int_{\R^d\times\R^d}|x-y|^\ell\rho(dx,dy)$. When $W_\ell(\mu,\nu)=0$ then the optimal coupling $\rho$ gives full weight to the diagonal $\{(x,x):x\in\R^d\}$ so that the two marginals $\mu$ and $\nu$ coincide. Since $\int_{\R^d\times\R^d}|x-y|^\ell\delta_x(dy)\mu(dx)=0$, where $\delta_x(dy)\mu(dx)$ is a coupling between $\mu$ and $\mu$, $W_\ell(\mu,\mu)=0$.
  To prove the triangle inequality we write disintegrations $\rho_y(dx)\nu(dy)$ and $\pi_y(dz)\nu(dy)$ of optimal couplings $\rho(dx,dy)$ and $\pi(dy,dz)$ between $\mu$ and $\nu$ and between $\nu$ and $\eta$. Then $\int_{y\in\R^d}\rho_y(dx)\pi_y(dz)\nu(dy)$ is a coupling between $\mu$ and $\eta$ and by subadditivity of $\R_+\ni u\mapsto u^\ell$,
\begin{align*}
   W_\ell(\mu,\eta)&\le \int_{(x,z)\R^d\times\R^d}|x-z|^\ell\int_{y\in\R^d}\rho_y(dx)\pi_y(dz)\nu(dy)\\&\le \int_{\R^d\times\R^d\times\R^d}|x-y|^\ell+|y-z|^\ell\rho_y(dx)\pi_y(dz)\nu(dy)=W_\ell(\mu,\nu)+W_\ell(\nu,\eta)
\end{align*}
so that the triangle inequality holds. Therefore $W_\ell$ is a metric on ${\cP}_\ell(\R^d)$.

Let $\underline W$ be defined as $W_1$ but with $|x-y|\wedge 1$ replacing the integrand $|x-y|$ in \eqref{defwl}. By Corollary 6.13 \cite{villani2008optimal}, $\underline W$ metricises the topology of weak convergence on ${\cP}_0(\R^d)$. Since for all $x,y\in\R^d$, $|x-y|\wedge 1\le |x-y|^\ell$, $\underline W\le W_\ell$. Morover, the second inequality in \eqref{triangsub} and the existence of an optimal coupling $\rho$ between $\mu$ and $\nu$ imply that
 \begin{align*}
  \left|\int_{\R^d}|x|^\ell\mu(dx)- \int_{\R^d}|y|^\ell\mu(dy)\right|&=\left|\int_{\R^d\times\R^d}\left(|x|^\ell- |y|^\ell\right)\rho(dx,dy)\right|\le\int_{\R^d\times\R^d}\left||x|^\ell- |y|^\ell\right|\rho(dx,dy)\\&\le \int_{\R^d\times\R^d}|x-y|^\ell\rho(dx,dy)=W_\ell(\mu,\nu).
 \end{align*}
Hence if $(\mu_n)_{n\in\N}$ is a sequence in $\cP_\ell(\R^d)$ such that $\lim_{n\to\infty}W_\ell(\mu_n,\mu)=0$, then $\mu_n$ converges weakly to $\mu$ as $n\to\infty$ and $\lim_{n\to\infty}\int_{\R^d}|x|^\ell\mu_n(dx)=\int_{\R^d}|x|^\ell\mu(dx)$. The converse implication can be checked by repeating the proof of the same statement for $\ell\ge 1$ p101-103 \cite{villani2008optimal}.

\end{proof}

\end{document}